\documentclass[12pt]{article}
\usepackage{amssymb,amsthm,amsmath,amsfonts,latexsym,tikz,hyperref}
\usepackage[hmargin=.8in,vmargin=.9in]{geometry}

\usetikzlibrary{patterns}

\newcommand{\ben}{\begin{enumerate}}
\newcommand{\een}{\end{enumerate}}
\newcommand{\ble}{\begin{lem}}
\newcommand{\ele}{\end{lem}}
\newcommand{\bth}{\begin{thm}}
\renewcommand{\eth}{\end{thm}}
\newcommand{\bpr}{\begin{prop}}
\newcommand{\epr}{\end{prop}}
\newcommand{\bco}{\begin{cor}}
\newcommand{\eco}{\end{cor}}
\newcommand{\bcon}{\begin{conj}}
\newcommand{\econ}{\end{conj}}
\newcommand{\bde}{\begin{defn}}
\newcommand{\ede}{\end{defn}}
\newcommand{\bex}{\begin{exa}}
\newcommand{\eex}{\end{exa}}
\newcommand{\barr}{\begin{array}}
\newcommand{\earr}{\end{array}}
\newcommand{\btab}{\begin{tabular}}
\newcommand{\etab}{\end{tabular}}
\newcommand{\beq}{\begin{equation}}
\newcommand{\eeq}{\end{equation}}
\newcommand{\bea}{\begin{eqnarray*}}
\newcommand{\eea}{\end{eqnarray*}}
\newcommand{\bal}{\begin{align*}}
\newcommand{\bce}{\begin{center}}
\newcommand{\ece}{\end{center}}
\newcommand{\bpi}{\begin{picture}}
\newcommand{\epi}{\end{picture}}
\newcommand{\bpp}{\begin{picture}}
\newcommand{\epp}{\end{picture}}
\newcommand{\bfi}{\begin{figure} \begin{center}}
\newcommand{\efi}{\end{center} \end{figure}}
\newcommand{\bprf}{\begin{proof}}
\newcommand{\eprf}{\end{proof}\medskip}
\newcommand{\capt}{\caption}
\newcommand{\bsl}{\begin{slide}{}}
\newcommand{\esl}{\end{slide}}
\newcommand{\bfr}{\begin{frame}}
\newcommand{\efr}{\end{frame}}

\newcommand{\hqed}{\hfill \qed}

\newcommand{\eqed}[1]{$\textcolor{white}{\qed}\hfill{\dil#1}\hfill\qed$}

\newcommand{\ul}{\underline}
\newcommand{\ol}{\overline}

\newcommand{\hso}[1]{\hspace{-1pt}}

\newcommand{\emp}{\emptyset}

\newcommand{\sbe}{\subseteq}

\newcommand{\spe}{\supseteq}

\newcommand{\ptn}{\vdash}

\newcommand{\gauss}[2]{\genfrac{[}{]}{0pt}{}{#1}{#2}}

\def\<{\langle}
\def\>{\rangle}

\newcommand{\ree}[1]{(\ref{#1})}

\newcommand{\ra}{\rightarrow}

\newcommand{\be}{\beta}
\newcommand{\ga}{\gamma}
\newcommand{\de}{\delta}

\newcommand{\la}{\lambda}

\newcommand{\si}{\sigma}

\newcommand{\cM}{{\cal M}}
\newcommand{\cN}{{\cal N}}

\DeclareMathOperator{\area}{area}

\DeclareMathOperator{\inc}{inc}

\DeclareMathOperator{\inv}{inv}

\DeclareMathOperator{\maj}{maj}

\DeclareMathOperator{\st}{st}

\DeclareMathOperator{\Par}{Par}

\newcommand{\dil}{\displaystyle}

\newtheorem{thm}{Theorem}[section]
\newtheorem{prop}[thm]{Proposition}
\newtheorem{cor}[thm]{Corollary}
\newtheorem{lem}[thm]{Lemma}
\newtheorem{conj}[thm]{Conjecture}
\newtheorem{exa}[thm]{Example}

\DeclareMathOperator{\lb}{lb}
\DeclareMathOperator{\ls}{ls}
\DeclareMathOperator{\rb}{rb}
\DeclareMathOperator{\rs}{rs}
\DeclareMathOperator{\lrs}{lrs}
\DeclareMathOperator{\LB}{LB}
\DeclareMathOperator{\LS}{LS}
\DeclareMathOperator{\RB}{RB}
\DeclareMathOperator{\RS}{RS}

\begin{document}
\pagestyle{plain}

\title{Restricted growth function patterns and statistics\thanks{All authors' research partially supported by NSA grant H98230-13-1-0259, by NSF grant DMS-1062817, and by the SURIEM REU at Michigan State University}\\[-5pt]
}
\author{Lindsey R. Campbell\\[-5pt]
\small Department of Procurement, Detroit Diesel Corporation\\[-5pt]
\small Detroit, MI 48239, USA, {\tt Lindsey.Reppuhn11@gmail.com} \\
Samantha Dahlberg\\[-5pt]
\small Department of Mathematics, Michigan State University,\\[-5pt]
\small East Lansing, MI 48824-1027, USA, {\tt dahlbe14@msu.edu}\\
Robert Dorward\\[-5pt]
\small Department of Mathematics, Oberlin College,\\[-5pt]
\small Oberlin, OH 44074, USA, {\tt bobbydorward@gmail.com}\\
Jonathan Gerhard\\[-5pt]
\small Department of Mathematics, James Madison University\\[-5pt]
\small Harrisonburg, VA 22801, USA, {\tt gerha2jm@dukes.jmu.edu} \\
Thomas Grubb\\[-5pt]
\small Department of Mathematics, Michigan State University,\\[-5pt]
\small East Lansing, MI 48824-1027, USA, {\tt grubbtho@msu.edu}\\
Carlin Purcell\\[-5pt]
\small Department of Mathematics, Vassar College, 124 Raymond Ave. Box 1114, \\[-5pt]
\small Poughkeepsie, NY 12604 USA, {\tt carlinpurcell@gmail.com}\\
Bruce E. Sagan\\[-5pt]
\small Department of Mathematics, Michigan State University,\\[-5pt]
\small East Lansing, MI 48824-1027, USA, {\tt sagan@math.msu.edu}
}

\date{\today\\[10pt]
	\begin{flushleft}
	                                       \\[5pt]
	\end{flushleft}}

\maketitle

\begin{abstract}
A restricted growth function (RGF) of length $n$ is a sequence $w=w_1 w_2\dots w_n$  of positive integers such that $w_1=1$ and $w_i \le 1+\max\{w_1,\dots, w_{i-1}\}$ for $i\ge2$.   RGFs are of interest because they are in natural bijection with set partitions of $\{1,2,\dots,n\}$.   RGF $w$ avoids  RGF $v$ if there is no subword of $w$ which standardizes to $v$.  We study the generating functions $\sum_{w\in R_n(v)} q^{\st(w)}$ where $R_n(v)$ is the set of  RGFs of length $n$ which avoid $v$ and $\st(w)$ is any of the four fundamental statistics on RGFs defined by Wachs and White.  These generating functions exhibit interesting connections with integer partitions and two-colored Motzkin paths, as well as noncrossing and nonnesting set partitions.
\end{abstract} 
 
%
%

\section{Introduction}

Recently, there has been a flurry of activity looking at the distribution of statistics over pattern classes in various objects.  For example, see~\cite{cddds:zas,ceks:ipt,ddjss:pps,ds:paa,gm:psq,gs:sps,kil:pcc}. There are two notions of pattern containment for set partitions, one obtained by standardizing a subpartition and one obtained by standardizing a subword of the corresponding restricted growth function.  In~\cite{ddggprs:spp}, the present authors studied the distribution of four fundamental statistics of Wachs and White~\cite{ww:pqs} over avoidance classes using the first definition.  The purpose of this paper is to carry out an analogous investigation for the second.

Let us begin by defining our terms.  Consider a finite set $S$.  A {\em set partition} $\si$ of $S$ is a family of nonempty subsets $B_1,\dots,B_k$ whose disjoint union is $S$, written $\si=B_1/\dots/B_k\ptn S$. The $B_i$ are called {\em blocks} and we will usually suppress the set braces and commas in each block for readability.  We will be particularly interested in set partitions of $[n]:=\{1,2,\dots,n\}$ and will use the notation
$$
\Pi_n=\{\si\ :\ \si\ptn[n]\}.
$$  
To illustrate $\si=145/2/3\ptn[5]$.  If $T\sbe S$ and $\si=B_1/\dots/B_k\ptn S$ then there is a corresponding {\em subpartion} $\si'\ptn T$ whose blocks are the nonempty intersections $B_i\cap T$.  To continue our example, if $T=\{2,4,5\}$ then we get the subpartition $\si'=2/45\ptn T$.

The concept of pattern is built on the standardization map.  Let $O$ be an object with labels which are positive integers.  The {\em standardization} of $O$, $\st(O)$, is obtained by replacing all occurrences  of the smallest label in $O$ by $1$, all occurrences of the next smallest by $2$, and so on.  Say that $\si\ptn[n]$ {\em contains $\pi$ as a pattern} if it contains a subpartition $\si'$ such that $\st(\si')=\pi$.  
In this case $\si'$ is called an {\em occurrence} or {\em copy} of $\pi$ in $\si$.
Otherwise, we say that $\si$ {\em avoids}  $\pi$ and let
$$
\Pi_n(\pi) = \{\si\in\Pi_n\ :\ \text{$\si$ avoids $\pi$}\}.
$$
In our running example, $\si=145/2/3$ contains $\pi=1/23$ since $\st(2/45)=1/23$.  But $\si$ avoids $12/3$ because if one takes any two elements from the first block of $\si$ then it is impossible to find an element from another block bigger than both of them.  Klazar~\cite{kla:aas,kla:cpsI,kla:cpsII} was the first to study this approach to set partition patterns.  For more recent work, see the paper of Bloom and Saracino~\cite{bs:pas}.

To define the second notion of pattern containment for set partitions, we need to introduce restricted growth functions.  A sequence $w=w_1w_2\dots w_n$  of positive integers is a {\em restricted growth function} (RGF) if it satisfies the conditions
\ben
\item  $w_1=1$, and
\item  for $i\ge2$ we have
\beq
\label{RGF}
w_i\le 1+\max\{w_1,\dots, w_{i-1}\}.
\eeq
\een
For example, $w=11213224$ is an RGF, but $w=11214322$ is not since $4>1+\max\{1,1,2,1\}$.  The number of elements of $w$ is called its {\em length}  and denoted $|w|$.  Define
$$
R_n=\{w\ :\ \text{$w$ is an RGF of length $n$}\}.
$$
For RGFs and even more general sequences we will use $w_i$ as the notation for the $i$th element of $w$.

To connect RGFs with set partitions, we will henceforth write all $\si=B_1/B_2/\dots/B_k\ptn[n]$ in {\em standard form} which means that  
$$
\min B_1 <\min B_2<\dots<\min B_k.
$$
Note that this implies $\min B_1=1$.  Given $\si=B_1/\dots/B_k\ptn [n]$ in standard form, we construct an associated word $w(\si)=w_1\dots w_n$ where
$$
\text{$w_i = j$ if and only if $i\in B_j$.}
$$
More generally, for any set $P$ of set partitions, we let $w(P)$ denote the set of $w(\si)$ for $\si\in P$.
Returning to our running example, we have $w(145/2/3)=12311$.  It is easy to see that $\si$ being in standard form implies $w(\si)$ is an RGF and  that the map $\si\mapsto w(\si)$ is a bijection $\Pi_n\ra R_n$.

We can now define patterns in terms of RGFs.  Given RGFs $v,w$ we call $v$ a {\em pattern} in  $w$ if there is a subword $w'$ of $w$ with $\st(w')=v$.  The use of the terms ``occurrence,"  ``copy," and ``avoids" in this setting are the same as for set partitions.   Given $v$ we define the corresponding {\em avoidance class}
$$
R_n(v)=\{w\in R_n\ :\ \text{$w$ avoids $v$}\}.
$$
Similarly define, for any set $V$ of RGFs,
$$
R_n(V)=\{w\in R_n\ :\ \text{$w$ avoids every $v\in V$}\}.
$$
As before, consider $w=w(145/2/3)=12311$.  Then $w$ contains $v=121$ because either of the subwords $121$ or $131$ of $w$ standardize to $v$.  However, $w$ avoids $v=122$ since the only repeated elements of $w$ are ones.  Note that this is in contrast to the fact that $145/2/3$ contains $1/23$ where $w(1/23)=122$.  In general, we have the following result whose proof is straight forward and left to the reader.
\bpr
\label{R-Pi}
Suppose that partitions $\pi$ and $\si$ have RGFs $v=w(\pi)$ and $w=w(\si)$.  If $w$ contains $v$ then $\si$ contains $\pi$, but not necessarily conversely.  Equivalently, we have $R_n(v)\spe w(\Pi_n(\pi))$.\hqed
\epr

Sagan~\cite{sag:pas} described the sets $R_n(v)$ for all $v\in R_3$.  To state this result, we need a few more definitions.  The {\em initial run} of an RGF $w$ is the longest prefix of the form $12\dots m$.  Write $a^l$ to indicate a string of $l$ copies of the integer $a$.  Call the word $w$ {\em layered} if it has the form $w=1^{n_1} 2^{n_2}\dots m^{n_m}$, equivalently, if it is weakly increasing.  Although the next result is not new, we include the proof since it is essential in much of what follows.
\bth[\cite{sag:pas}]
\label{rgf:len3}
We  have the following characterizations.
\ben
\item $R_n(111)=\{ w\in R_n\ :\ \text{every element of $w$ appears at most twice}\}$.
\item $R_n(112)=\{ w\in R_n\ :\  \text{$w$ has initial run $12\dots m$ and $m\ge w_{m+1}\ge w_{m+2}\ge\dots\ge w_n$}\}$.
\item $R_n(121)=\{ w\in R_n\ :\ \text{$w$ is layered}\}$.
\item $R_n(122)=\{ w\in R_n\ :\ \text{every element $j\ge 2$ of $w$ appears only once}\}$.
\item $R_n(123)=\{ w\in R_n\ :\ \text{$w$ contains only $1$s and $2$s}\}$.
\een
\eth
\bprf
In all cases, it is easy to see that the $w$ described on the right-hand side are in the avoidance class.  So we will concentrate on proving the reverse containments.  This is most easily done by contradiction.

\medskip

1.  If $w$ has some $j$ appearing three or more times then $jjj$ is an occurrence of $111$ in $w$.

\medskip

2.  If $w$ does not have this form, then consider the first index $i\ge m$ such that $w_i<w_{i+1}$.  By definition of initial run, $i>m$.  So, since this is the first such index, there must be $w_h=w_i$ for some $w_h$ in the initial run.  But then 
$w_h w_i w_{i+1}$ is a copy of $112$.

\medskip

3.   If $w$ is not layered, then consider its longest layered prefix $w_1\dots w_i$.  Since $w$ is an RGF and the prefix is maximal, we must have $i\ge2$ and $w_i>w_{i+1}$.  Similar to the previous proof, there must be $w_h$ in the layered prefix with $w_h=w_{i+1}$.  This makes $w_h w_i w_{i+1}$ is a copy of $121$.

\medskip

4.  Suppose some $j\ge2$ appears more than once.  Since every RGF begins with a $1$, we have an occurence $1jj$ of $122$.

\medskip

5.  Condition~\ree{RGF} implies that if $w$ contains a number larger than $2$ then it must contain a $3$ and, in fact, the subword $123$.
\eprf

Using this result, it is not hard to compute the cardinalities of  the  classes.
\bco[\cite{sag:pas}]
We have
$$
\#R_n(112)=\#R_n(121)=\#R_n(122)=\#R_n(123)=2^{n-1}
$$
and
$$
\#R_n(111)=\sum_{i\ge0} \binom{n}{2i}(2i)!!
$$
where $(2i)!!=1\cdot 3\cdot 5 \cdot\ \dots\ \cdot (2i-1)$.\hqed
\eco

Our object, in part, is to prove generalizations of the formulae in this corollary using the statistics of Wachs and White and their generating functions.  They defined four statistics on RGFs denoted $\lb$, $\ls$, $\rb$, and $\rs$ where the letters l, r, b, and s stand for left, right, bigger, and smaller, respectively.  We will explicitly define the $\lb$ statistic and the others are defined analogously.  Given a word $w=w_1 w_2\dots w_n$, let
$$
\lb(w_j)=\#\{w_i\ :\ \text{$i<j$ and $w_i>w_j$}\}.
$$
Otherwise put, $\lb(w_j)$ counts the number of integers which are to the left of $w_j$ in $w$ and bigger than $w_j$.  Note that multiple copies of the same integer which is left of and bigger than $w_j$ are only counted once.  Note, also, that $\lb(w_j)$ also depends on $w$ and not just the value of $w_j$.  But context will ensure that there is no confusion.  For an example, if $w=12332412$ then for $w_5=2$ we have $\lb(w_5)=1$ since three is the only larger integer which occurs before the two.  For $w$ itself, define
$$
\lb(w)=\lb(w_1)+\lb(w_2)+\dots+\lb(w_n).
$$
Continuing our example,
$$
\lb(12332412)=0+0+0+0+1+0+3+2=6.
$$
Finally, given an RGF, $v$, we consider the generating function
$$
\LB_n(v)=\LB_n(v;q)=\sum_{w\in R_n(v)} q^{\lb(w)}
$$
and similarly for the other three statistics.  Sometimes we will be able to prove things about multivariate generating functions such as
$$
F_n(v)=F_n(v;q,r,s,t)=\sum_{w\in R_n(v)} q^{\lb(v)} r^{\ls(v)} s^{\rb(v)} t^{\rs(v)}.
$$

As noted in Proposition~\ref{R-Pi}, if $v=w(\pi)$ then we always  have $R_n(v)\spe w(\Pi_n(\pi))$.  But for certain $\pi$ we have equality.  In particular, as shown in~\cite{sag:pas}, this is true for $\pi=123, 13/2, 1/2/3$ and the corresponding $v=111,121,123$.  So these patterns will not, for the most part, be analyzed in what follows since their generating functions were computed in~\cite{ddggprs:spp}.

The rest of this paper is organized as follows.  In the next section, we consider the generating functions for the two remaining patterns of length three, namely $v=112$ and $122$.   Interestingly, integer partitions and $q$-binomial coefficients come into play as well as connections with the polynomials for the remaining patterns of length three.
In Section~\ref{mpl}, we consider the classes $R_n(V)$ for all $V\sbe R_3$ containing two or more patterns.  The next three sections deal with RGFs of length longer than three.  Section~\ref{rfl} gives recursive methods for calculating generating functions for longer patterns in terms of shorter ones.  The following two sections are concerned with the patterns $1212$ and $1221$ which are connected with noncrossing and nonnesting set partitions, respectively.  These results are closely related to two-colored Motzkin paths.  We end with a section of comments and open questions.


\section{Single patterns of length $3$}
\label{spl}

\subsection{Patterns related to integer partitions in a rectangle}
In this subsection, we show bijectively that three of the generating functions under consideration are the same. Moreover, the common value can be expressed in terms of the $q$-binomial coefficients which count integer partitions in a rectangle.
 First we need some definitions about $q$-analogues and integer partitions.

We let $[n]_q = 1+ q +q^2 +\dots +q^{n-1}$. We can now define a $q$-analogue of the factorial, letting $[n]_q! = [1]_q[2]_q\cdots[n]_q$. Finally, we define the \emph{$q$-binomial coefficients} or \emph{Gaussian polynomials} as
$$ \gauss{n}{k}_q = \frac{[n]_q!}{[k]_q![n-k]_q!}.$$
By convention, $\gauss{n}{k}_q = 0$ if $k<0$ or $k>n$.

A \emph{partition $\lambda = (\lambda_1,\lambda_2,\dots,\lambda_k) $ of an integer} $t$ is a weakly decreasing sequence of positive integers such that $\sum _{i=1}^k \lambda_i = t$. We call the $\lambda_i$ \emph{parts} and let $|\lambda| = \sum _{i=1}^k \lambda_i$. 
We let $\Par$ denote the set of all integer partitions.
The \emph{Young diagram} of a partiton $\lambda$ is an array of boxes with $k$ left-justified rows, where the $i$th row has $\lambda_{i}$ boxes. For example the partition $\lambda = (5,5,4,3,3)$ would correspond to the Young diagram in Figure~\ref{yd}. Sometimes we will need to refer to particular boxes in the Young diagram. We let $(i,j)$ denote the box in the $i$th row and $j$th column.
\begin{figure}
\begin{center}
\begin{tikzpicture}[scale = .6]
\draw (0,0) rectangle (5,1);
\draw (0,1) rectangle (1,-4);
\draw (2,0) --  (2,1);
\draw (3,0) --  (3,1);
\draw (4,0) --  (4,1);
\draw (0,-1) --  (1,-1);
\draw (0,-2) --  (1,-2);
\draw (0,-3) --  (1,-3);
\draw (1,-1) rectangle (5,0);
\draw (1,0) rectangle (2,-4);
\draw (3,-1) --  (3,0);
\draw (4,-1) --  (4,0);
\draw (1,-2) --  (2,-2);
\draw (1,-3) --  (2,-3);
\draw (2,-2) rectangle (4,-1);
\draw (2,-1) rectangle (3,-4);
\draw (2,-3) --  (3,-3);
 \end{tikzpicture}
\end{center}
\capt{The Young diagram for $\lambda=(5,5,4,3,3)$ \label{yd}}
\end{figure}

If $\beta$ is an $r\times \ell$ rectangle, written $\be=r\times\ell$, and $\lambda= (\lambda_1,\lambda_2,\dots,\lambda_k)$ is a partition, we say that the Young diagram of $\lambda$ \emph{fits inside} $\beta$ if $k\leq r$ and $\lambda_1\leq \ell$.  We will denote this by $\lambda\subseteq \beta$. When $\lambda\subseteq \beta$, we will draw $\la$ and $\be$ so superimposed so that their $(1,1)$ boxes coincide, as can be seen in Figure~\ref{ydr}.
For $\beta=r\times\ell$ it is well-known that
$$\gauss{r+\ell}{\ell}_q = \sum_{\lambda\subseteq\beta} q^{|\lambda|}.$$

Now that we have the proper terminology, we can state our first equidistribution theorem.
\begin{figure}
\begin{center}
\begin{tikzpicture}[scale = .6]
\draw (0,0) rectangle (5,1);
\draw (0,1) rectangle (1,-4);
\draw (2,0) --  (2,1);
\draw (3,0) --  (3,1);
\draw (4,0) --  (4,1);
\draw (0,-1) --  (1,-1);
\draw (0,-2) --  (1,-2);
\draw (0,-3) --  (1,-3);
\draw (1,-1) rectangle (5,0);
\draw (1,0) rectangle (2,-4);
\draw[step=1.0,black,dotted] (0,-5) grid (5,0);
\draw (3,-1) --  (3,0);
\draw (4,-1) --  (4,0);
\draw (1,-2) --  (2,-2);
\draw (1,-3) --  (2,-3);
\draw (2,-2) rectangle (4,-1);
\draw (2,-1) rectangle (3,-4);
\draw (2,-3) --  (3,-3);
\draw (-1.5,-1);
 \end{tikzpicture}
\end{center}
\capt{The Young diagram for $\lambda=(5,5,4,3,3)$ in the $6\times 5$ rectangle $\beta$ \label{ydr}}
\end{figure}
\bth\label{connectLB112} We have
$$ \LB_n(112) = \RS_n(112) = \LB_n(122) =\sum_{t\ge 0} \gauss{n-1}{t}_q.$$
\eth
In other words, each of the above polynomials is the generating function for integer partitions counted with multiplicity given by the number of rectangles into which they fit.
We will establish Theorem \ref{connectLB112} through four propositions.

First, we need a few more definitions. A sequence of integers  $u_1\dots u_n$ is called \emph{unimodal} if there exists an index $i$ with
$$
u_1\le u_2\le \dots \le u_i \ge u_{i+1}\ge \dots \ge u_n.
$$
We define a \emph{rooted unimodal composition} $u = u_1\ldots \boldsymbol{u_m}\dots u_n$ to be a sequence of nonnegative integers, together with a distinguished element called the {\em root} and displayed in boldface type, having the following properties:

\medskip

1. $u$ is unimodal. 

\medskip

2. $u_1=u_n = 0$.

\medskip

3. If $u$ is rooted at $\boldsymbol{u_m}$, then $\boldsymbol{u_m} = \max(u)$.

\medskip 

4.  We have $|u_j-u_{j+1}|\leq 1$ for all $j$. 

\medskip

We define $|u| = u_1+\dots+u_n$ and let
$$A_n=\{u: u=u_1\dots u_n \text{ is a rooted unimodal composition}\}.$$
 For example $u = 00012\boldsymbol{2}22110000$ is a rooted unimodal composition with root $\boldsymbol{u_6}=\boldsymbol{2}$ 
and $|u| = 11$. The rooted unimodal compositions are useful to show that 
$\RS_n(112)$ is a sum of Gaussian polynomials.
\bpr
We have 
$$\RS_n(112) = \sum_{u\in A_n} q^{|u|}.$$
\epr
\begin{proof} It suffices to construct a bijection $\psi: R_n(112)\rightarrow A_n$ such that $\rs(w) = |\psi(w)|$. Let 
$w=w_1\dots w_n$, and let $m$ be the index of the first maximum of $w$. We will construct $\psi(w)=u= u_1\ldots \boldsymbol{u_m} \dots u_n$ by letting $u_i = \rs(w_i)$. For example if $w= 1234553221$ then $\psi(w) = 0123\boldsymbol{3}32110$.

We begin by showing $\psi$ is well-defined. Let $w\in R_n(112)$ and $u=\psi(w)$. By Theorem \ref{rgf:len3}, $w$ has some initial run $12\dots m$ which is followed by a weakly decreasing sequence with terms at most $m$. We will show $u$ satisfies properties 1--4 above.

First, properties 1 and 3 follow from the fact that $w$ is unimodal with maximum $m$. Because $w_1=1$ and $w_n$ is the right-most element, we have $\rs(w_1)=\rs(w_n)=0$ and thus, property 2. 
The fourth property holds because before the maximum index $m$, adjacent elements increase by one, and after that index the sequence is weakly decreasing.

We now define $\psi^{-1}$. Let $u = u_1\ldots \boldsymbol{u_m} \dots u_n \in A_n$. Let $\ell(j)$ be the index of the last occurrence of $j$ in $u_1\ldots \boldsymbol{u_m}$. We construct $w= \psi^{-1}(u)$ so that 
$$w=123\dots m w_{m+1}\dots w_n$$
where for $m+1\leq i\leq n$ we have $w_i = \ell(u_i)$. For example if $u=001122\boldsymbol{2}21000$ then $w= 123456774222$. To show $\psi^{-1}$ is well-defined, it suffices to show $w_i\geq w_{i+1}$ for $i\ge m$.  But this follows since $u_1\ldots \boldsymbol{u_m}$ is a weakly increasing sequence and so $u_i\ge u_{i+1}$ for $i\ge m$ implies $\ell(u_i)\ge\ell(u_{i+1})$.

Next, we show that the two maps are indeed inverses. First, assume $\psi(w) = u$ and $\psi^{-1}(u) = v=v_1\dots v_n$. Let $m$ be the index of the first maximum in $w$ so that $\boldsymbol{u_m}$ is the root in $u$. We will show $w_i=v_i$ for all $1\leq i\leq n$. We know $w_i=i$ for all $i\leq m$. Since $\boldsymbol{u_m}$ is rooted in $u$ then, by definition of $\psi^{-1}$, we have that $v$ also begins with $12\dots m$.  For $i>m$, there must be an index $k\le m$ with $k=w_k=w_i$  If follows that $u_k=\rs(w_k)=\rs(w_i)=u_i$.  Furthermore, $k$ must be the largest index less than or equal to $m$ which satisfies the last equality since $w_{k+1}=w_k+1$ so that $u_{k+1}>u_i$.  It follows, by definition of $\ell$, that $v_i=\ell(u_i)=k=w_i$.
The proof that $\psi(\psi^{-1}(u))=u$ is similar. 

If $\psi(w)= u$ then $u_i = \rs(w_i)$ and so $\rs(w)=|u|$. Therefore 
$$\RS_n(112) = \sum_{u\in A_n} q^{|u|}$$
as desired.
\end{proof}

Let 
$$B_n = 
\bigcup_{m\ge1} \{(\lambda, \beta): \text{$\lambda\in\Par$, $\beta=(m-1)\times (n-m)$,  and $\lambda\subseteq\beta$}\}.$$
As discussed above,
$$\sum_{(\lambda,\beta)\in B_n} q^{|\lambda|} = \sum_{m\geq 1} \gauss{n-1}{m-1}_q = \sum_{t\geq 0} \gauss{n-1}{t}_q.$$
\bpr\label{connect_unim}
We have 
$$\sum_{u\in A_n} q^{|u|} = \sum_{t\geq 0} \gauss{n-1}{t}_q.$$
\epr
\begin{proof} From the discussion just before this proposition, it suffices to construct a bijection $\varphi: A_n\rightarrow B_n$ such that if $u\in A_n$ and $\varphi(u)=(\lambda, \beta)$ then $|u|=|\lambda|$. Let $u=u_1\ldots\boldsymbol{u_m}\dots u_n\in A_n$. Then we construct $\varphi(u) = (\lambda,\beta)$ as follows. First, we use the index of the root of $u$ to determine that $\beta$ will be a $(m-1)\times (n-m)$ rectangle. Consider the diagonal in $\beta$ formed by coordinates $(1,1), (2,2),\dots$ and parallel diagonals above and below this one. Then, going from southwest to northeast, take the first $u_i$ squares along each diagonal as $i$ varies from $2$ to $n-1$ to form the diagram for $\lambda$.
 For example the rooted unimodal composition $u=001233\boldsymbol{3}32210$ will give $\lambda = (5,5,4,3,3)$ in the $6\times 5$ rectangle $\beta$ shown in Figure~\ref{ydr}.

Properties 1, 2, and 4 ensure that $\lambda = (\lambda_1,\dots, \lambda_k)$ is a well-defined Young diagram corresponding to an integer partition. Now we must check that $\lambda\subseteq\beta$. Using the ordering in which we constructed the diagonals of $\lambda$, observe that the number of nonempty diagonals up to and including the main diagonal is $k$, the number of parts of $\la$.   As the main diagonal  of $\lambda$ corresponds to element $\boldsymbol{u_m}$ and we begin in our construction with element $u_2$, we have $k\leq m-1$. Similarly, we can see that $\lambda_1$ is equal to the number of nonempty diagonals after and including the main diagonal. Thus $\lambda_1 \leq n-m$, as we end with element $u_{n-1}$. Therefore $\lambda\subseteq\beta$.

To define the inverse map note that if $(\lambda,\beta)\in B_n$, where $\beta$ is an $(m-1)\times (n-m)$ rectangle, then the root of $\varphi^{-1}(\lambda,\beta)$ should  be at index $m$. The entries of $\varphi^{-1}(\lambda,\beta)$ are obtained using the diagonals of $\lambda$ so as to reverse the above construction. As the construction of $\varphi^{-1}$ is very similar to that of $\varphi$, we leave the process of checking that $\varphi^{-1}$ is well defined and the inverse of $\varphi$ to the reader. In addition, it is clear from the definitions that $|u|=|\lambda|$. 
\end{proof}

Combining the two previous propositions shows that $\RS_n(112)$ has the desired form.
It should be mentioned that we originally proved Proposition \ref{connect_unim}  using a bijection involving hook decompositions, similar to Section 3 of the paper of Barnabei et al.~\cite{bbes:dsi}. Although the above proof was found to be simpler,  it may be interesting to further explore connections between patterns in RGFs and hook decompositions.
\bpr
\label{LB112}
We have
$$\LB_n(112)=\sum_{t\geq 0} \gauss{n-1}{t}_q .$$
\epr
\begin{proof}We will construct a bijection $\rho:R_n(112)\rightarrow B_n$ such that $\lb(w)= |\lambda|$, where 
$w\in R_n(112)$ and $\rho(w)=(\lambda,\beta)$. Let the initial run of $w$ be $w_1\dots w_m=1 2\dots m$ so that $m=\max(w)$.

First, we let $\beta$ be an $(n-m)\times (m-1)$ rectangle. We then let 
$$\lambda = (m-w_n,m-w_{n-1},\dots,m-w_{m+1}), $$ 
permitting parts equal to zero. For example, $w=123456633211$ would map to $(\lambda,\beta)$ shown in Figure~\ref{ydr}.

As $m$ is the maximum of $w$, we have $0\leq m-w_i\leq m-1$ for $m+1\leq i\leq n$. In addition, $w_{m+1}\geq w_{m+2}\geq \dots \geq w_n$ and thus the parts of $\lambda$ are weakly decreasing. Therefore $\lambda$ is well-defined and fits inside $\beta$. Constructing $\rho^{-1}$ is a simple matter which we leave to the reader.

Now notice that in $w$, $\lb(w_i) = 0$ for all $1\leq i\leq m$. For $m<i\leq n$, we have that $\lb(w_i) = m-w_i = \lambda_{n-i+1}$. Thus 
$$
\lb(w)= \lb(w_{m+1})+\lb(w_{m+2})+\dots+\lb(w_{n}) = \lambda_{n-m} +\lambda_{n-m-1}+\dots+\lambda_1 = |\lambda|
$$
as desired.
\end{proof}
\bpr\label{connectLB122}
We have
$$\LB_n(112) = \LB_n(122).$$
\epr
\begin{proof}We will construct a bijection $\eta:R_n(112)\rightarrow R_n(122)$ such that $\lb(w) = \lb(\eta(w))$. Let 
$w\in R_n(112)$ have maximum $m$. To construct $\eta(w)$ we start with the sequence $12\dots m$. For every $w_i$, where $w_i$ is not in the initial run of $w$, we will insert a 1 just to the right of element $m-w_i+1$ in $\eta(w)$. Note that $1\leq w_i\leq m$ ensures that this element always exists, and in conjunction with Theorem~\ref{rgf:len3} this shows that $\eta$ is well-defined. For example if $w=12345664331$ then $\eta(w) = 11231411561$.
Clearly $\eta$ is invertible.

To check that $\lb$ is preserved, note that in $w$ the initial run does not contribute to $\lb$ and in $\eta(w)$, none of the terms greater than 1 contribute to $\lb$. Consider $w_i$ such that $i> m$. Then $\lb(w_i) = m-w_i$. If we examine the 1 placed into $\eta(w)$ because of $w_i$, we notice that it has $m-w_i$ terms greater than 1 to its left. Therefore the $\lb$ of this 1 is $m-w_i$. Thus, $\lb(w)=\lb(\eta(w))$.
\end{proof}
Combining the above propositions yields Theorem \ref{connectLB112}.

\subsection{Patterns related to integer partitions with distinct parts}
Next, we will explore a connection to integer partitions with distinct parts. It is well-known that the generating function for partitions with distinct parts of size at most $n-1$ is 
$$\prod_{i=1}^{n-1} (1+q^i).$$
As noted in the introduction, for the pattern $121$ we have $R_n(121) = w(\Pi_n(13/2))$.  So we can use the following result of Goyt and Sagan who studied the $\ls$ statistic on $\Pi_n(13/2)$.
\bpr[\cite{gs:sps}] We have 
$$\LS_n(121) =\RB_n(121)= \prod_{i=1}^{n-1} (1+q^i).$$
\epr
The following result establishes that, once again, four of our generating functions are the same.
\bth\label{connectLS112RB122} We have the equalities
$$\LS_n(112) = \LS_n(121)  =\RB_n(121)=\RB_n(122) = \prod_{i=1}^{n-1} (1+q^i).$$
\eth

As before, we break the proof of this result into pieces.

\bpr\label{connectLS112} We have 
$$\LS_n(112) = \LS_n(121).$$
\epr
\begin{proof} 
We will construct a bijection $\xi: R_n(112)\rightarrow R_n(121)$ such that $\ls(w) = \ls(\xi(w))$. Given $w\in R_n(112)$ we will construct $\xi(w)$ by rearranging the elements of $w$ in weakly increasing order.  By Theorem~\ref{rgf:len3}, this is well defined.   For the inverse, if we are given a layered RGF, $v$, then we use the first element of each layer to form an initial run and rearrange the remaining elements in weakly decreasing order.

For any RGF $w=w_1 \dots w_n$ we have $\ls(w_i) = w_i-1$.  Since $w$ and $\xi(w)$ are rearrangements of each other, $\ls$ is preserved.
\end{proof}
\bpr\label{connectRB122} We have 
$$\LS_n(112) = \RB_n(122).$$
\epr
\begin{proof}Let $\eta:R_n(112)\rightarrow R_n(122)$ be as in Proposition \ref{connectLB122}.
To see that $\ls(w) = \rb(\eta(w))$, first note that $\ls(w_i) = w_i -1$. By construction, the initial run of $w$ has $\ls$ that is equal to the total $\rb$ of the first occurrences of elements in $\eta(w)$. In addition, for each $w_i$ not in the initial run of $w$, we place a 1 to the right of $m-w_i+1$ in $\eta(w)$, and therefore there are $w_i-1$ elements to its right that are larger than it. Thus $\ls(w) = \rb(\eta(w))$.
\end{proof}

Combining the above propositions, we obtain Theorem \ref{connectLS112RB122}.

\subsection{Patterns not related to integer partitions}

In this section, we present two more connections between the generating functions of patterns of length 3.  The first is as follows.

\bth
\label{RS122LBRS123}
We have 
$$\RS_n(122) = \LB_n(123) = \RS_n(123) = 1+ \sum_{k=0}^{n-2}\binom{n-1}{k+1}q^k.$$
\eth
\begin{proof} 
It was shown in~\cite{ddggprs:spp} that
$$
\LB_n(123) = \RS_n(123) = 1+ \sum_{k=0}^{n-2}\binom{n-1}{k+1}q^k.
$$
So it suffices to construct a bijection $f:R_n(122)\rightarrow R_n(123)$ which preserves the $\rs$ statistic.
First, recall that by Theorem \ref{rgf:len3}, words in $R_n(123)$ contain only $1$s and $2$s and that for $w\in R_n(122)$, every element $j\geq 2$ of $w$ appears only once. Given $w=w_1\dots w_n\in R_n(122)$, we will construct 
$f(w)=u_1\dots u_n$ by replacing each element $j\geq 2$ in $w$ with a 2. This is a bijection, as any word in $R_n(122)$ is uniquely determined by the placement of its ones. In addition, $\rs(w_i)=\rs(u_i)$ by construction so that $\rs(w)=\rs(f(w))$.
\end{proof}

The second establishes yet another connection between statistics on $R_n(112)$ and $R_n(122)$.
\bth\label{connectLS122} We have 
$$\RB_n(112)=\LS_n(122)= \sum_{m=0}^n\binom{n-1}{n-m}q^{\binom{m}{2}}.$$
\eth

\begin{proof} For the first equality, let $\eta:R_n(112)\rightarrow R_n(122)$ be as in Propositions \ref{connectLB122} and \ref{connectRB122}. We will show that for $w\in R_n(112)$, we have $\rb(w) = \ls(\eta(w))$. Because $w$ is unimodal, only the initial run contributes to $\rb$. If $m$ is the largest element in the inital run of $w$, then $\rb(w) = 1+ 2+\dots +(m-1)= \binom{m}{2}$. Similarly, only the elements greater than 1 in $\eta(w)$ contribute to $\ls$. By construction, the largest element in $\eta(w)$ is $m$ as well. Thus, $\ls(\eta(w)) = 1+2+\dots +(m-1) = \binom{m}{2}$. 

To show that $\RB_n(112)= \sum_m \binom{n-1}{n-m}q^{\binom{m}{2}}$ it suffices, as can be seen from the previous paragraph, to count the number of $w\in R_n(112)$ with  initial run $12\dots m$.   Notice that once the elements in the weakly decreasing sequence following the initial run have been selected, there is only one way to order them.  For that sequence we must choose $n-m$ elements from the set $[m]$, allowing repetition, yielding a total of $\binom{n-1}{n-m}$ as desired.
\end{proof}
It is remarkable that the map $\eta$ connects so many of the statistics on $R_n(112)$ and $R_n(122)$; see the the proofs of Propositions~\ref{connectLB122}, \ref{connectRB122}, and Theorem~\ref{connectLS122}.  The four-variable generating functions $F_n(v;q,r,s,t)$ can be used to succinctly summarize these demonstrations as follows.
\bth
We have

\medskip

\eqed{F_n(112; q, r, s, 1) = F_n(122; q, s, r, 1).}
\eth


\section{Multiple patterns of length 3}
\label{mpl}

This section considers RGFs which avoid multiple patterns  of length three. In all cases we are able to determine the four-variate generation function. We find connection to Gaussian polynomials, integer partitions, and Fibonacci numbers.

For sets $V\subseteq R_3$ it is not hard to see if $121\in V$ or $\{111,122\}\subseteq V$ then  the two notions of pattern avoidance, as RGFs and as set partitions, are equivalent. In these cases the characterization and cardinality of $R_n(V)$  have been determined by Goyt~\cite{goy:apt}, and the generating functions $F_n(V)$ have been determined by~\cite{ddggprs:spp}.  For completeness we will include the characterization, cardinality, and generating function for all $V\subseteq R_3$ except those which contain both $111$ and $123$ since in these cases $R_n(V)=\emp$  for $n\geq 5$.

The following characterizations are obtained directly by taking the intersection of the sets described in Theorem~\ref{rgf:len3} so the proof is omitted.

\bth
\label{multichar}
We  have the following characterizations.
\ben
\item $R_n(111,112)=\{ w\in R_n\ :\   \text{$w$ has initial run $12\dots m$ and $m\geq a_{m+1}>a_{m+2}>\dots> a_n$}\}$.
\item $R_n(111,121)=\{ w\in R_n\ :\    \text{$w$ is layered and every element of $w$ appears at most twice}      \}$.
\item $R_n(111,122)=\{ w\in R_n\ :\  \text{$w=12\dots n$  or  $w=12\dots i1(i+1)\dots (n-1)$ for some $0< i<n$}           \}$.
\item $R_n(112,121)=\{ w\in R_n\ :\     \text{$w=12\dots mm\dots m$ for some $1\le m\le n$}   \}$.
\item $R_n(112,122)=\{ w\in R_n\ :\   \text{$w=12\dots m11\dots 1$ for some $1\le m\le n$}        \}$.
\item $R_n(112,123)=\{ w\in R_n\ :\     w=12^{i}1^{n-i-1}   \text{ for some $0\leq i <n$}     \}$.
\item $R_n(121,122)=\{ w\in R_n\ :\    \text{$w=1^{i}23\dots (n-i+1)$ for some $0< i\le n$}         \}$.
\item $R_n(121,123)=\{ w\in R_n\ :\          \text{$w=1^i2^{n-i}$ for some $0< i\le n$}   \}$.
\item $R_n(122,123)=\{ w\in R_n\ :\      w=1^n \text{ or } w=1^{i}21^{n-i-1}\text{ for some $0< i <n$}       \}$.
\een
\label{thm:doublechr}
\eth

Let $f_n$ denote the $n$th Fibonacci number defined by $f_0=f_1=1$ and $f_n=f_{n-1}+f_{n-2}$ for $n\ge2$.

\begin{cor} The cardinality of the avoidance sets above are as follows. 
\begin{enumerate}
\item $\# R_n(111,112)=\#R_n(111,121)=f_n$. 
\item For all other pairs  of length 3 patterns $\{v_1,v_2\}$ except for the pair $\{111,123\}$ we have $\#R_n(v_1,v_2)=n$. 
\end{enumerate}
\end{cor}
\begin{proof}
All the cardinalities have been determined previously by Goyt~\cite{goy:apt}  or can be determined easily from Theorem~\ref{thm:doublechr} except $\# R_n(111,112)$. 

We will now show that $\# R_n(111,112)$ satisfies the Fibonacci recurrence. It is not hard to see that $\#R_0(111,112)=\#R_1(111,112)=1$. Next consider $w \in R_n(111,112)$ with $n\geq 2$.  We know from Theorem~\ref{thm:doublechr} that $w=12\dots ma_{m+1}\dots a_n$ where $m\geq a_{m+1}>\dots > a_n$ which implies that $w$ has either one $1$ or two $1$'s.  If $w$ has one $1$ then that $1$ is at the beginning and $w=1(v+1)$ for some $v\in R_{n-1}(111,112)$. If $w$ has two $1$'s then  the second $1$ will be $a_n$ and $w=1(v+1)1$ for some $v\in R_{n-2}(111,112)$. This gives us the desired recurrence $\#R_n(111,112)=\#R_{n-1}(111,112)+\#R_{n-2}(111,112)$.
\end{proof}

All the sets described in Theorem~\ref{thm:doublechr} are sufficiently simple that we can determine   their four-variable generating functions. Many of the functions can be simplified by extending the Gaussian polynomials which were defined at the beginning of Section~\ref{spl} to two variables. The bivariate analogue of  $n$ is
$$[n]_{p,q}=p^{n-1}+p^{n-2}q+\dots +q^{n-1}$$
so $[n]_{p,q}!=[n]_{p,q}[n-1]_{p,q}\dots [1]_{p,q}$ and the binomial analogue is
$${n\brack k}_{p,q}=\frac{[n]_{p,q}!}{[k]_{p,q}![n-k]_{p,q}!}.$$
We can recover the one-variable version by letting $p=1$. 

In Section~\ref{spl} we introduced a classic  interpretation for the $q$-binomial coefficients in terms of integer partitions. There is a similar well-known interpretation  in the bivariate case.  Before we give the interpretation we need one definition. 
Given $\beta$ a $r \times \ell$ box and a partition $\lambda=(\lambda_1,\dots \lambda_k) \subseteq \beta$ we define its {\it complement} $\lambda^c$  to be the partition which is composed of all boxes of $\lambda$'s Young diagram in $\beta$ outside $\lambda$ rotated $180^{\circ}$. 
 Continuing our example from Figure~\ref{ydr}, the partition $\lambda=(5,5,4,3,3)$ in the $6\times 5$ box  has $\lambda^c=(5,2,2,1)$ as its compliment. 
Note that  $|\lambda^c|=r\ell-|\lambda|$. For $\beta$, a $r\times \ell$ box, we have the well-known formula
$${r+\ell \brack \ell}_{p,q}=\sum_{\lambda \subseteq \beta}p^{|\lambda|}q^{|\lambda^c|}.$$

Almost all the functions $F_n(V)$ for $V=\{v_1,v_2\}\subset R_3$ are computed in~\cite{ddggprs:spp}, \cite{gs:sps}, or follow easily from methods previously used in this paper, so the proofs will be omitted.  The only exception is $V=\{111,112\}$ for which we will provide a demonstration. The method we use parallels the proof used for $F_{n}(111,121)$ in~\cite{gs:sps}.
\bth
We  have the following generating functions.
\label{multFn}
\ben
\item $\displaystyle F_n(111,112)= \sum_{m\geq 0}   (qrt^2)^{\binom{m}{2}}  (rs)^{\binom{n-m}{2}}  {n-m \brack m}_{r,qt}$.

\item $\displaystyle F_n(111,121)=\sum_{m\geq 0}(rs)^{\binom{m}{2}+\binom{n-m}{2}}{n-m \brack m }_{r,s}$.

\item $\displaystyle F_n(111,122)=(rs)^{\binom{n}{2}}+(rs)^{\binom{n-1}{2}}[n-1]_{s,qt}$.

\item $\displaystyle F_n(112,121)=\sum_{m=1}^nr^{(m-1)(n-m)}(rs)^{\binom{m}{2}}$.

\item $\displaystyle F_n(112,122)=(rs)^{\binom{n}{2}}+\sum_{m=1}^{n-1}q^{(m-1)(n-m)}(rs)^{\binom{m}{2}}t^{m-1}$.

\item $\displaystyle F_n(112,123)=1+r^{n-1}s +qrst[n-2]_{q,rt}$.

\item $\displaystyle F_n(121,122)=\sum_{m=1}^{n}(rs)^{\binom{m}{2}}s^{(m-1)(n-m)}$.

\item $\displaystyle F_n(121,123)=1+rs[n-1]_{r,s}$.

\item $\displaystyle F_n(122,123)=1+rs^{n-1}+qrst[n-2]_{q,s}$. 

\een
\eth
\bprf
Let $\rho:R_n(112)\ra B_n$ be the map of Proposition~\ref{LB112}.  We will continue to use the other notation in the proof of that result.  If we restrict $\rho$ to $R_n(111,112)$ then, by Theorem~\ref{multichar}, the image of the restricted map is
$$
C_n = 
\bigcup_{m\ge1} \{(\lambda, \beta): \text{$\lambda\in\Par$ has distinct parts, $\beta=(n-m)\times (m-1)$,  and $\lambda\subseteq\beta$}\}.
$$
If $w\mapsto (\lambda,\beta)$ then we claim
$$
(\lb(w),\ls(w),\rb(w),\rs(w))= \left(|\la|,\ \binom{m}{2}+|\la^c|,\  \binom{m}{2},\ \binom{n-m}{2}+|\la| \right).
$$
Indeed, $\lb(w)=|\la|$ was proved in Proposition~\ref{LB112}.  For $\ls(w)$, the binomial coefficient comes from the initial run, while for $i>m$ we have
$$
\ls(w_i)=w_i-1=(m-1)-(m-w_i)=\la_{i-m}^c.
$$
Only the initial run contributes to $\rb(w)$, giving $\binom{m}{2}$.  Call the subword $w_{m+1} w_{m+1}\dots w_n$ of $w$ its {\em tail}.  Since the tail is strictly decreasing, it will contribute $\binom{n-m}{2}$ to $\rs(w)$.  If $w_i$ is in the initial run, then $\rs(w_i)$ is the number of elements in the tail smaller than $w_i$.  But this is the same as the number of boxes in column $m-i+1$ of $\la$ and so the initial run adds another $|\la|$ to $\rb(w)$.

There is a standard bijection $\de$ from partitions with $r$ distinct parts $\la=(\la_1,\dots,\la_r)$ to ordinary partitions 
$\mu=(\mu_1,\dots,\mu_r)$ with $r$ parts where in both case we permit zero as a part. It is given by 
$$
\de(\la_1,\la_2,\dots,\la_{r-1},\la_r)=(\la_1-(r-1),\la_2-(r-2),\dots,\la_{r-1}-1,\la_r)=\mu.
$$
Note the $|\la|=|\mu|+\binom{r}{2}$.  Also,
if $\la\sbe r\times \ell$ then $\mu\sbe r\times (\ell-r+1)$. 
 Furthermore $\mu^c=\de(\la^c)$.

Now if  $\la\sbe (n-m)\times(m-1)$ then $\de(\la)\sbe (n-m)\times(2m-n)$.  It follows that we have a bijection $\rho':R_n(111,112)\ra C_n'$ where
$$
C_n'= 
\bigcup_{m\ge1} \{(\mu, \ga): \text{$\mu\in\Par$, $\ga=(n-m)\times (2m-n)$,  and $\mu\subseteq\ga$}\}.
$$
Furthermore, if $\rho'(w)=(\mu,\ga)$ then
$$
\barr{l}
(\lb(w),\ls(w),\rb(w),\rs(w))\\[10pt]
\hspace{30pt}\dil = \left(\binom{n-m}{2}+|\mu|,\ \binom{m}{2}+\binom{n-m}{2}+|\mu^c|,\  \binom{m}{2},\  
2\binom{n-m}{2}+|\mu|\right).
\earr
$$
Translating this bijection into a generating function identity and then replacing $m$ by $n-m$ yields the desired equation.
\eprf

From the functions provided in Theorem~\ref{multFn} we can see several symmetries and invariants. 

\begin{cor}
\label{sym}
 Let $V\subseteq R_3$ and $n\geq 0$.
\begin{enumerate}
\item For $V=\{v_1,v_2\}$ such that $121\in V$ or $V=\{111,122\}$, the function $F_n(V)$ is invariant under switching $q$ and $t$.
\item The following sets $V$ have $F_n(V)$ invariant under switching  $r$ and $s$. 
$$\{111,121\}, \{112,122\},\{121,123\}.$$
\item We have the  equalities
$$F_n(111,121;q,r,s,t)=F_n(111,112;s,r,s,1),$$
$$F_n(112,121;q,r,s,t)=F_n(121,122;q,s,r,t),$$
and
$$F_n(112,123;q,r,s,1)=F_n(122,123;q,s,r,1).$$
\end{enumerate}
\end{cor}

The avoidance classes for $V\subseteq R_3$ of size three and four are easy to determine and so we will merely list them in  Table~\ref{tab:multi}.  The reader interested in the corresponding generating functions will be able to easily write them down.

\begin{table}
\def\arraystretch{1.5}
\begin{center}
\begin{tabular}{|c|c|}
\hline
$V$  & $R_n(V)$\\
\hline
$\{111,112,121\}$&$12\dots n,\ 12\dots (n-2)(n-1)^2$\\
\hline
$\{111,112,122\}$& $12\dots n,\ 12\dots (n-1)1$\\
\hline
$\{111,121,122\}$&$12\dots n,\ 1123\dots (n-1)$\\
\hline
$\{112,121,122\}$&$12\dots n,\ 1^n$\\
\hline
$\{112,121,123\}$&$1^n,\ 12^{n-1}$\\
\hline
$\{112,122,123\}$&$1^n,\ 121^{n-2}$\\
\hline
$\{121,122,123\}$&$1^n,\ 1^{n-1}2$\\
\hline
\hline
$\{111,112,121,122\}$&$12\dots n$\\
\hline
$\{112,121,122,123\}$&$1^n$\\
\hline
\end{tabular}
\end{center}
\caption{Avoidance classes for $V\subset R_3$ of size three and four and $n\ge3$. }
\label{tab:multi}
\end{table}


\section{Recursive Formulae and Longer Words}
\label{rfl}

In this section we will investigate generating functions for avoidance classes of various RGFs of length greater than three.  This includes a recursive formula for computing the generating funtions for longer words in terms of shorter ones.

Let $w+k$ denotes the word obtained by adding the nonnegative integer $k$ to every element of $w$. Note that if $w$ is an RGF and $k$ is nonzero, then $w+k$ will not be an RGF. However, the word $\bar{w} = 12\dots k(w+k)$ obtained by concatenating the increasing sequence $12\dots k$ with $w+k$, will be an RGF.   In fact, there is a relationship between the generating functions for $w$ and $\bar{w}$. In the following theorem, we show that this relationship holds for the $\ls$ and $\rs$ statistic.
We note that in \cite[Propositions 2.1 and 2.2]{ms:paspcn}, 
Mansour and Shattuck use the same method to find the cardinalities of the avoidance classes of the pairs of patterns $\{1222, 12323\}$ and $\{1222,12332\}$.

\bth \label{LSRSmachine}
Let $v$ be an RGF and $\bar{v} = 1(v+1).$ Then
$$ \LS_n(\bar{v}) = \sum_{j =0}^{n-1} \binom{n-1}{j}q^j\LS_j(v) $$
and
$$\RS_n(\bar{v}) = \sum_{j=0}^{n-1} \sum_{k=0}^j \binom{n+k-j-2}{k} q^k \RS_j(v).$$
\eth

\bprf
We start by building the avoidance class of $\bar{v}$ out of the avoidance class of $v$. We do so by taking a word $w$ in the avoidance class of $v$, forming $1(w+1)$, and then adding a sufficient number of ones to $1(w+1)$ to obtain a word $\bar{w}$ of length $n$ which avoids $\bar{v}$. We then count how adding these ones affects the respective statistics. 

We first establish that avoidance is preserved in this process. Let $w \in R_j(v)$. Since $w$ avoids $v$, we know $1(w+1)$ avoids $1(v+1) = \bar{v}$. Now we need to show that forming  $\bar{w}$ by adding $n-j-1$ ones to  $1(w+1)$ in any manner will result in $\bar{w}$ avoiding $\bar{v}$. If $\bar{w} \not\in R_n(\bar{v})$, then there is a subword $w^{\prime}$ of $w$ such that $\st(w^{\prime}) = \bar{v}.$ Since $\bar{v} = 1(v+1)$, the smallest element of $w^{\prime}$ must appear only at the beginning of the subword, and must be a $1$ since $1(w+1)$ avoided $\bar{v}$. But removing the unique $1$ and standardizing the remaining elements shows that there is a subword of $w$ that standardizes to $v$. This is a contradiction. Therefore, we must have $\bar{w} \in R_n(\bar{v}).$ Similarly, every word in $R_n(\bar{v})$ with $n-j$ ones can be turned into a word in $R_j(v)$ by removing all ones and standardizing. If this word wasn't in $R_j(v)$, then it would contain a subword that standardized to $v$. As before, this would mean the original word contained $1(v+1) = \bar{v}$, which is a contradiction. Therefore, we can construct every word in $R_n(\bar{v})$ from the words in $R_j(v)$ for $j \in [0,n-1]$.

We now translate this process into the generating function identities.  First we will focus on the $\LS$ formula. We can choose any $w \in R_j(v)$, and place the elements of $w+1$ in our word $\bar{w}$ in $\binom{n-1}{j}$ different ways since we must leave the first position free to be a one. Then we fill in the rest of the positions with ones. Since we added $1$ to each element of $w \in R_j(v)$ and added a one to the beginning of the word, we have $\ls(\bar{w}) = \ls(w) + j$. So
$$\LS_n(\bar{v}) = \sum_{\bar{w} \in R_n(\bar{v})}q^{\ls(\bar{w})}  = \sum_{j =0}^{n-1}\sum_{w \in R_j(v)} \binom{n-1}{j}q^jq^{\ls(w)} = \sum_{j =0}^{n-1} \binom{n-1}{j}q^j\LS_j(v). $$

For the $\RS$ formula, instead of all $j$ elements of $w+1$ increasing the statistic, only the $k$ elements of $w+1$ that are to the left of the rightmost one in $\bar{w}$ will contribute. If we choose where to place these elements, then everything else is forced. We start with $n-1$ positions available, and disregard $j-k+1$ for the rightmost one and the elements of $w+1$ that appear after it. Thus we have $(n-1) - (j-k+1) = n+k-j-2$ positions to choose from. Summing over all values of $j$ and $k$ gives the $\RS$ formula.
\eprf

In the paper of Dokos et al.~\cite{ddjss:pps}, the authors introduced the notion of statistical Wilf equivalence.
We will consider how this idea can be applied to the four statistics we have been studying.
 We define two RGFs $v$ and $w$ to be {\em $\ls$-Wilf-equivalent} if $\LS_n(v) = \LS_n(w)$ for all $n$, and denote this by
$$
v \overset{\ls}{\equiv} w.
$$
Similarly define an equivalence relation for the other three statistics. 
Let $\st$ denote any of our four statistics.  Given any equivalence $v \overset{\st}{\equiv} w$, we can  generate an infinite number of related equivalences.
\bco
Suppose $v \overset{\st}{\equiv} w$.   Then for any $k\ge1$ we have
$$
12\dots k (v+k) \overset{\st}{\equiv} 12\dots k (w+k).
$$
\eco
\bprf
For $\st=\ls,\rs$ this follows immediately from Theorem~\ref{LSRSmachine} and induction on $k$.  For the other two statistics, note that the same ideas as in the proof of Theorem~\ref{LSRSmachine} can be used to show that one can write down the generating function for $\st$ over $R_n(12\dots k(v+k))$ in terms of the the generating functions for $\st$ over $R_j(v)$ for $j\le n$ although the expressions are more complicated.  Thus induction can also be used in these cases as well.
\eprf
Applying this corollary to the equivalences in Theorem~\ref{connectLB122},  Proposition~\ref{connectLS112}, and  Theorem~\ref{RS122LBRS123} yields the following result.
\bco
We have 
\begin{align*}
12\dots kk(k+1)  &\overset{\lb}{\equiv}  12\dots k(k+1)(k+1),\\
12\dots kk(k+1)  &\overset{\ls}{\equiv}  12\dots k(k+1)k,\\
12\dots k(k+1)(k+1)  &\overset{\rs}{\equiv}  12\dots k(k+1)(k+2),
\end{align*}
for all $k\ge1$.\hqed
\eco

We will now demonstrate how these ideas can be used to find the generating functions for a family of RGFs by finding $\LS_n(12\dots k)$ for a general $k$. We begin by finding the degree of $\LS_n(12\dots k)$ through a purely combinatorial approach before using Theorem~\ref{LSRSmachine} to give a formula for the generating function itself.

\bpr\label{Prop:DegLS}
For $n\ge k$, the generating function $\LS_n(12\dots k)$ is monic and
$$
\deg \LS_n(12\dots k)=\binom{k-2}{2} + (k-2)(n-k+2).
$$
\epr

\bprf
It is easy to see that $w\in R_n(12\dots k)$ if and only if $w_i<k$ for all $i$.  Also $\ls(w_i)=w_i-1$ for all $i$.  Thus there is a unique word maximizing $\ls$, namely 
$w = 12\dots (k-2)(k-1) \dots (k-1)$.  Thus $\LS_n(12\dots k)$ is monic with 
$\ls(w) = 0 + 1 + 2 + \dots + (k-2) + (n-k+1)(k-2) = \binom{k-1}{2} + (k-2)(n-k+1)$.
\eprf

To obtain a formula for $L_n(12\dots k)$ we will use the $q$-analogues introduced earlier, often suppressing the subscript $q$ for readability.  Consider the rational function of $q$
$$K_{m,n} = \frac{[m+1]^{n-1} - 1}{[m]}.$$ 
We will need the following facts about $K_{m,n}$.  Writing $[m+1]^{n-1}=(1 + q[m])^{n-1}$ and expanding by the binomial theorem gives
\begin{equation}
\label{Kmn}
K_{m,n}=\sum_{j =1}^{n-1} \binom{n-1}{j}q^j[m]^{j-1}.
\end{equation}
We also have
\begin{equation}
\label{Kmndiff}
\frac{1}{[m]}(K_{m+1,n} - K_{1,n})=\sum_{j =1}^{n-1} \binom{n-1}{j}q^jK_{m,j} 
\end{equation}
which can be obtained by substituting the definition of $K_{m,j}$ into the sum and then applying the previous equation.

Finally we define, for $k\ge3$,
$$c_k = 1 - \sum_{j=1}^{k-3} \frac{1}{[j]!}c_{k-j}.$$
Note that when $k=3$ the sum is empty and so $c_3=1$.  Note also that for fixed $k$, the number of terms in $LS_n(12\dots k)$ is a linear function of $n$ by Proposition~\ref{Prop:DegLS}.  However, in the formula for this generating function which we give next the number of summands only depends on $k$, making it an efficient way to compute this polynomial.

\bth\label{LS(12...k)}
For $k\ge 3$, we have 
$$
\LS_n(12\dots k) = 1 + \sum_{i=1}^{k-2}\frac{1}{[i-1]!}c_{k-i+1}K_{i,n}.
$$
\eth

\bprf

We proceed with a proof by induction. In~\cite{ddggprs:spp}, the authors show that $\LS_n(1/2/3) = [2]^{n-1}$ for the set partition $1/2/3$. Recall that a set partition avoids $1/2/3$ if and only if its corresponding RGF avoids $123$. Therefore $\LS_n(1/2/3) = \LS_n(123) = [2]^{n-1}$  for $n\ge1$. Rewriting this as $\LS_n(123) = 1 + K_{1,n}$ gives our base case for $k = 3$.

Suppose the equality held for $ k  \ge 3$. Then, using Theorem~\ref{LSRSmachine} as well as equations~(\ref{Kmn}) and~(\ref{Kmndiff}),
\begin{align*}
\LS_n(12\dots k+1) &= 1 + \sum_{j=1}^{n-1}\binom{n-1}{j}q^j\LS_j(12\dots k) \\
                                &= 1 + \sum_{j=1}^{n-1}\binom{n-1}{j}q^j\left(1 + \sum_{i=1}^{k-2}\frac{1}{[i-1]!}c_{k-i+1}K_{i,j} \right)\\
                                &= 1 + \sum_{j=1}^{n-1}\binom{n-1}{j}q^j +  \sum_{i=1}^{k-2}\frac{1}{[i-1]!}c_{k-i+1}\left(\sum_{j=1}^{n-1}\binom{n-1}{j}q^jK_{i,j} \right)\\
                                &= 1 + K_{1,n} +  \sum_{i=1}^{k-2}\frac{1}{[i]!}c_{k-i+1}(K_{i+1,n} - K_{1,n})\\
                                &= 1 + K_{1,n}\left(1 - \sum_{i=1}^{k-2}\frac{1}{[i]!}c_{k-i+1}\right) +  \sum_{i=1}^{k-2}\frac{1}{[i]!}c_{k-i+1}K_{i+1,n}\\
			&= 1 + c_{k+1}K_{1,n} +  \sum_{i=1}^{k-2}\frac{1}{[i]!}c_{k-i+1}K_{i+1,n}\\
                                &= 1 + \sum_{i=1}^{k-1}\frac{1}{[i-1]!}c_{k-i+2}K_{i,n}
\end{align*}
which completes the induction.
\eprf

Let $1^m$ denote the RGF consisting of $m$ copies of one.  The ideas in the proof of Theorem~\ref{LSRSmachine} can be used to give  recursive formulae for this pattern.  It would be interesting to find other patterns where this reasoning could be applied.
\bth
For $m \ge 0$, we have
$$LS_n(1^m) = \sum_{j=1}^{m-1} \binom{n-1}{j-1}q^{n-j}\LS_{n-j}(1^m)$$
and
$$\RS_n(1^m) = \RS_{n-1}(1^m)+\sum_{j=2}^{m-1} \sum_{k=0}^{n-j}\binom{j+k-2}{k}q^k\RS_{n-j}(1^m).$$
\eth
\bprf
Let $w$ avoid $1^m$.  Then $w$ can be uniquely obtained by taking a $w'$ avoiding $1^m$ and inserting $j$ ones in $w'+1$, where $1\le j\le m-1$ and a one must be inserted at the beginning of the word.  The formula for $\LS_n(1^m)$ now follows since the binomial coefficient counts the number of  choices for the non-initial ones, $\LS_{n-j}(1^m)$ is the contribution from $w'+1$, and $q^{n-j}$ is the obtained from the interaction between the initial one and $w'+1$.  The reader should now have no problem modifying the proof of the $\RS_n(\bar{v})$ formula in Theorem~\ref{LSRSmachine} to apply to this case.
\eprf


\section{The pattern $1212$}
\label{p1212}

\subsection{Noncrossing partitions}
The set partitions which avoid the pattern $13/24$ are called \emph{non-crossing} and are of interest, in part, because of their connection with Coxeter groups and the Catalan numbers
$$
C_n=\frac{1}{n+1}\binom{2n}{n}.
$$  See the memoir of Armstrong~\cite{arm:gnp} for more information.  In this case the set containment in Proposition~\ref{R-Pi} can be turned into an equality as we will show next.  Note that $w(13/24)=1212$.
\bpr
\label{noncross_partition_rgf}
We have 
$$
R_n(1212)=w(\Pi_n(13/24)).
$$
\epr
\bprf
As just noted, it suffices to show that if $\pi$ contains $13/24$, then $w(\pi)$ contains $1212$. By definition, if $\pi$ contains $13/24$, then $w(\pi)$ contains a subword $xyxy$ for some $x\neq y$. If $x<y$, then this will standardize to $1212$ as desired. If $x>y$ then, because $w(\pi)$ is a restricted growth function, there must be some occurrence of $y$ before the leftmost occurrence of $x$ in $w(\pi)$. Thus $w(\pi)$ also contains a subword $yxyx$ which is a copy of $1212$ in $w(\pi)$. 
\eprf

With this proposition in hand, we now focus on gaining information about these partitions by studying $R_n(1212)$. We begin by applying the $\rs$ statistic to $R_n(1212)$, and in doing so obtain a $q$-analogue of the standard Catalan recursion. 
We first need the following lemma regarding $1212$-avoiding restricted growth functions.
\ble
\label{noncross_avoid}
For an RGF $w$, the following are equivalent:
\begin{enumerate}
\item[(1)] The RGF $w$ avoids $1212$.
\item[(2)] There are no $xyxy$ subwords in $w$. 
\item[(3)] If $w_i=w_{i`}$ for some $i<i'$ then, for all $j'>i'$, either  $w_{j'}\leq w_{i'}$ or $w_{j'}>\max\{w_1,\dots,w_{i'}\}$ .
\end{enumerate}
\ele
\bprf
The equivalence of the first two statements follows from the proof of Proposition~\ref{noncross_partition_rgf}. It thus suffices to show that $(2)$ and $(3)$ are equivalent. First, let $w$ be an RGF with no $xyxy$ subword, and let $w_i=w_{i'}$ for some $i<i'$. Assume, towards contradiction, that there exists a $j'$ with $j'>i'$ and $w_{i'}< w_{j'}\leq\max\{w_1,\dots w_{i'}\}$. This implies that there exists a $j$ with $j<i'$ and $w_j=w_{j'}$. If $i<j<i'$, then $w_iw_jw_{i'}w_{j'}$ forms an $xyxy$ subword in $w$, a contradiction. If this is not the case, then since $w_j>w_i$ and $w$ is an RGF, there must exist another occurance of the letter $w_i$ preceding $w_j$. This letter, combined with $w_j$, $w_{i'}$, and $w_{j'}$ still creates an $xyxy$ subword, which is again a contradiction. This shows that $(2)$ implies $(3)$.

Now we show that if $w$ contains an $xyxy$ subword, then $w$ cannot satisfy $(3)$. Indeed, by the discussion in the proof of Proposition~\ref{noncross_partition_rgf}, if $w$ contains an $xyxy$ subword then, without loss of generality, we may assume $x<y$. Thus the second occurrence of $y$ in the subword will violate condition $(3)$. This completes the proof of the equivalence of the statements.
\eprf

We now move to a recursive way of producing words in $R_n(1212)$. 
\bco
\label{noncross_avoid_cor}
If $u$ is in $ R_{n-1}(1212)$ then both $1u$ and $1(u+1)$ are in $R_n(1212)$.
\eco
\bprf
Let $u$ be an element of $R_{n-1}(1212)$. By the previous lemma, we know that $u$ does not contain any $xyxy$ subwords. Prepending a $1$ to $u$ will not create any such subword, as otherwise this would imply an $xyxy$ subword in $u$ using its leading $1$. Therefore $1u$ is contained in $R_n(1212)$. Furthermore, adding one to each element in $u$ to create $u+1$ will not introduce an $xyxy$ subword, and prepending a $1$ to create $1(u+1)$ will not create an $xyxy$ subword as there is only one copy of $1$ in $1(u+1)$. Thus $1(u+1)$ is also contained in $R_n(1212)$. 
\eprf
With these results in hand, we move to one of the main results of this section. For two words $w$ and $u$, we will use the set notation $w\cap u=\emptyset$ to denote that $w$ and $u$ have no elements in common. The next theorem gives a $q$-analogue of the usual recursion for the Catalan numbers. It will also be used to establish a connection between $R_n(1212)$ and lattice paths. 
\bth
\label{rs1212}
We have 
\begin{align*}
\RS_0(1212)&=1,\\
\RS_1(1212)&=1,
\end{align*}
and for $n\geq2$, 
$$
\RS_n(1212)=2\RS_{n-1}(1212)+\sum_{k=1}^{n-2}q^k\RS_k(1212)\RS_{n-k-1}(1212).
$$
\eth
\bprf
The base cases are trivial.
To prove the recursion, we partition $R_n(1212)$ into three disjoint subsets $X$, $Y$, and $Z$ as follows:
\begin{align*}
X&=\{w\in R_n(1212)\ :\ w_1=1 \text{ and there are no other 1s in }w\},\\
Y&=\{w\in R_n(1212)\ :\ w_1w_2=11\},\\
Z&=\{w\in R_n(1212)\ :\ w_1w_2=12 \text{ and there is at least one other 1 in } w\}.
\end{align*}
We claim that we can also describe $X$ as the set of words defined by 
\beq
X=\{w=1(u+1)\ :\ u\in R_{n-1}(1212)\}.
\eeq
To see this, let $u$ be a word in $R_{n-1}(1212)$. From Corollary~\ref{noncross_avoid_cor}, we know $w=1(u+1)$ is an element of $R_n(1212)$, and by definition of $u+1$, the only $1$ in $w$ will be $w_1$. This gives one containment. Now let $w$ be an element of $X$ as originally defined.  Since the leading one in $w$ is unique, let $u+1$ denote the last $n-1$ letters in $w$. By Lemma~\ref{noncross_avoid}, $w$ contains no $xyxy$ subword; in particular, $u+1$ contains no $xyxy$ subword. Standardizing $u+1$ to the RGF $u$ will not create any $xyxy$ subwords, and thus $u$ will be contained in $R_{n-1}(1212)$. This gives the reverse containment, from which we conclude that the two sets are equal. A similar proof, without standardization of the subword, allows us to describe $Y$ as the set 
\beq
Y=\{w=1u\ :\ u\in R_{n-1}(1212)\}.
\eeq

Now note that for any RGF $u$, we have $\rs(u)=\rs(1(u+1))$ and $\rs(u)=\rs(1u)$. Using this fact, and the above characterization of the sets, we can see that $X$ and $Y$ must contribute $\RS_{n-1}(1212)$ each to the total $\RS_n(1212)$ polynomial. 

Finally, we claim that we can characterize $Z$ as 
\begin{align}
Z=\{w=1(u+1)1v\ :\ &u\in R_k(1212) \text{ for } 1\leq k\leq n-2,\nonumber\\
&\st(1v)\in R_{n-k-1}(1212), v\cap (u+1)=\emptyset\}.
\end{align}
First, let $w$ be contained in $Z$ as defined at the beginning of the proof.  By definition of $Z$, $w$ has a nonempty subword of the form $u+1$ consisting of all entries between the first and second $1$ in $w$. Let the length of $u$ be $k$. As with the set $X$, $u+1$ will standardize to $u$, an RGF in $R_k(1212)$. Now let $v$ be the last $n-k-2$ letters in $w$, so that our word is of the form 
$$
w=1(u+1)1v.
$$
Since a $1$ is repeated before $v$, we must have $v_i=1$ or $v_i>\max(u+1)$ for all $i$ by Lemma~\ref{noncross_avoid}, where $v_i$ is the $i$th letter of $v$. This gives $v\cap(u+1)=\emptyset$. Furthermore, there is no $xyxy$ subword contained in $1v$, and standardizing the subword will not create an $xyxy$ pattern. Thus $\st(1v)$ is contained in $R_{n-k-1}(1212)$. This shows one inclusion between the two versions of $Z$. Now let $u$ be an element of $R_k(1212)$, and let $1v'$ be an element of $R_{n-k-1}(1212)$. Corollary~\ref{noncross_avoid_cor} gives that $1(u+1)$ avoids $1212$ as well. Now from the RGF $1v'$, we create the word $1v$ by setting
$$
(1v)_i=
\begin{cases}
\begin{array}{lc}
(1v')_i &\text{ if }(1v')_i=1\\
(1v')_i+\max(u) &\text{ if }(1v')_i\neq1.
\end{array}
\end{cases}
$$
We claim that $w=1(u+1)1v$ is a member of $R_n(1212)$. To see this, note that $u+1$ contains no $xyxy$ subwords, and further $u+1$ shares no integers in common with the rest of $w$. Therefore $u+1$ cannot contribute to an $xyxy$ subword in $w$. Thus if such a subword existed in $w$, it must also exist in $11v$. This is impossible as it would imply an $xyxy$ subword in $1v'$, contradicting our choice of $1v'$. We have now shown the reverse set containment, which implies the desired equality of the two sets. 

With this characterization of $Z$, we can now decompose $\rs(w)$ for $w$ in $Z$ as 
$$
\rs(w)=\rs(u+1)+k+\rs(1v),
$$
where the middle term comes from the contribution to $\rs$ caused by comparing the elements of $u+1$ with the second $1$ in $w$. Summing over all possibilities of $k$, $u$, and $v$, and noting that the $\rs$ of a word is not affected by standardization, we can see that $Z$ will contribute 
$$
\sum_{k=1}^{n-2}q^k\RS_k(1212)\RS_{n-k-1}(1212).
$$
Adding the results obtained from $X$, $Y$, and $Z$ now gives the desired total. 
\eprf

 \begin{figure}
  \begin{center}
\begin{tikzpicture}[scale = .8]
\draw (0,0) --  (1,1) -- (2,1) -- (3,2) -- (4,2)--(5,1) --(6,2) --(7,2) --(8,1) --(9,0) --(10,1) --(11,1)--(12,0)  ;
\draw [dashed]  (0,0)--(12,0);
\draw [dashed]  (2,1) -- (8,1);
\draw [dashed]  (1,1) -- (1,0);
\draw [dashed]  (2,1) -- (2,0);
\draw [dashed]  (3,2) -- (3,0);
\draw [dashed]  (4,2) -- (4,0);
\draw [dashed]  (5,1) -- (5,0);
\draw [dashed]  (6,2) -- (6,0);
\draw [dashed]  (7,2) -- (7,0);
\draw [dashed]  (8,1) -- (8,0);
\draw [dashed]  (10,1) -- (10,0);
\draw [dashed]  (11,1) -- (11,0);
\draw (1.5,1.3) node {$b$};
\draw (3.5,2.3) node {$a$};
\draw (6.5,2.3) node {$b$};
\draw (10.5,1.3) node {$b$};
\draw(0.5,-.4) node{$s_1$};
\draw(1.5,-.4) node{$s_2$};
\draw(2.5,-.4) node{$s_3$};
\draw(3.5,-.4) node{$s_4$};
\draw(4.5,-.4) node{$s_5$};
\draw(5.5,-.4) node{$s_6$};
\draw(6.5,-.4) node{$s_7$};
\draw(7.5,-.4) node{$s_8$};
\draw(8.5,-.4) node{$s_9$};
\draw(9.5,-.4) node{$s_{10}$};
\draw(10.5,-.4) node{$s_{11}$};
\draw(11.5,-.4) node{$s_{12}$};
 \end{tikzpicture} 
 \caption{A two-colored Motzkin path}
  \label{tcm}
   \end{center}
 \end{figure}
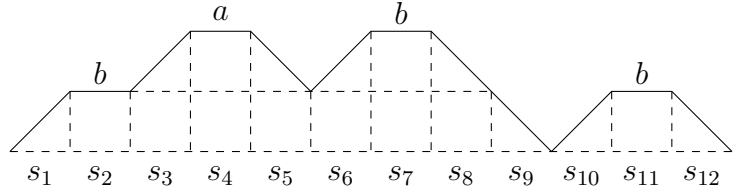

For the next result, we first recall the definition of a Motzkin path. A Motzkin path $P$ of length $n$ is a lattice path in the plane which starts at $(0,0)$, ends at $(n,0)$, stays weakly above the $x$-axis, and which uses vector steps in the form of up steps (1,1), horizontal steps (1,0), and down steps (1,-1).  Let $\cM_n$ denote the set of all Motzkin paths of length $n$.
We write $P=s_1\dots s_n$ for such a path, where 
$$
s_i=
\begin{cases}
U\text{ if the }i\text{th step is an up step},\\
H\text{ if the }i\text{th step is a horizontal step},\\
D\text{ if the }i\text{th step is a down step}.\\
\end{cases}
$$
Given a step $s_i$ in $P$, we can realize $s_i$ geometrically as a line segment in the plane connecting two lattice points in the obvious way. 
Figure~\ref{tcm} displays the Motzkin path $P=UHUHDUHDDUHD$.
Define the \emph{level} of $s_i$, $l(s_i)$, to be the lowest $y$-coordinate in $s_i$.   Continuing with our example path, the sequence of levels of the steps is $0,1,1,2,1,1,2,1,0,0,1,0$.
Note that the level statistic provides a natural pairing of up steps with down steps in a Motzkin path. Namely, we associate an up step $s_i$ with the first down step $s_j$, $j>i$, which is at the same level as $s_i$, i.e. $l(s_i)=l(s_j)$. We will call such steps \emph{paired}.  In Figure~\ref{tcm} the pairs are $s_1$ and $s_9$, $s_3$ and $s_5$, $s_6$ and $s_8$,  and $s_{10}$ and $s_{12}$.

We now define a \emph{two-colored Motzkin path} $P$ of length $n$ to be a Motzkin path of length $n$ 
whose horizontal steps are individually colored using one of the colors $a$ or $b$. We will call an $a$-colored horizontal step an \emph{$a$-step} and a $b$-colored horizontal step a \emph{$b$-step}. For a two-colored Motzkin path $P = s_1 \dots s_{n}$ we will still use $s_i$ equal to $U$ or $D$ for up steps and down steps, but will use $a$ or $b$ instead of $H$ to show the color of the horizontal steps.  In this notation, our example path is $P=UbUaDUbDDUbD$.    Let $\cM^2_n$ denote the set of all two-colored Motzkin paths of length $n$. For two paths $P=s_1\dots s_{n}$ and $Q=t_1\dots t_{m}$ we write $PQ=s_1\dots s_{n} t_1\dots t_{m}$ to indicate their concatenation.  Interestingly, Wachs and White were originally inspired to look at the RGF statistics because of a question posed by   Dennis Stanton (personal communication)   and motivated by the appearance two-colored Motzkin paths  in a combinatorial interpretation of the moments of $q$-Charlier polynomials given by Viennot~\cite{vie:ctg}.

Let the \emph{area} of a path $P$, $\area(P)$, denote the area enclosed between $P$ and the $x$-axis. 
Our example  has $\area(P)=14$.
Defining
\begin{equation}
\label{M_n(q)}
M_n(q)=\sum_{P\in \cM_n^2}q^{\area(P)},
\end{equation}
Drake~\cite{dra:laulp} proved the following recursion.
\bth[\cite{dra:laulp}]
\label{drake}
We have $M_0(q)=1$ and, for $n\geq1$,
$$
M_n(q)=2M_{n-1}(q)+\sum_{k=1}^{n-2}q^kM_k(q)M_{n-k-1}(q).
$$
\qed
\eth
Using Theorems~\ref{rs1212} and~\ref{drake} as well as induction on $n$ immediately gives the following equality.
\bco
\label{rs1212=M}
We have
$$
\RS_n(1212)=M_{n-1}(q)
$$
for all $n\ge1$.\hqed
\eco

Interestingly, it turns out that we also have $\LB_n(1212)=\LB_n(1221)=M_{n-1}(q)$ which will be proved in Section~\ref{p1221}.
In our next result, we prove the previous corollary directly via a bijection between $\cM_{n-1}^2$ and $R_n(1212)$.  Before doing so, 
it will be useful to discuss left to right maxima.  A sequence $w$ of integers has a \emph{left to right maximum} at $i$ if  $w_i>\max\{w_1,\dots, w_{i-1}\}$.  If $w$ is an RGF then clearly the left to right maxima occur exactly when $w_i=1+\max\{w_1,\dots, w_{i-1}\}$.  Another characterization for RGFs is that $w$ has  a left to right maximum at $i$ if and only if $w_i$ is the first occurrence of that value in $w$.

\bth
\label{1212_with_2cmotz}
There is an explicit  bijection $\psi:\cM_{n-1}^2\to R_n(1212)$ such that $\area(P)=\rs(\psi(P))$ for all $P\in \cM_{n-1}^2$.
\eth

\bprf
Given $P=s_1\dots s_{n-1}\in \cM_{n-1}^2$ we define $\psi(P)=w=w_1\dots w_{n}$ 
as follows. Let $w_1=1$ and 
$$
w_{i+1} = \left\{
\begin{array}{ll}
1+\max\{w_1,\dots, w_{i}\} &\text{ if  $s_i$ equals $U$ or $b$},\\
w_{i}& \text{ if $s_i=a$},\\
w_{j}& \text{ if $s_i=D$ is paired with the up step $s_j$}.
\end{array}
\right.
$$
Continuing with our example, $\psi(UbUaDUbDDUbD)=1234435631781$.  
We first show that $\psi$ is well defined.
By definition we have $w_1=1$ and,  for $i>1$, $w_i$  is either equal to $w_j$ for some $j<i$ or $1+\max\{w_1,\dots, w_{i-1}\}$. This implies that that $w_i$ is a positive integer and $w_i\leq 1+\max\{w_1,\dots, w_{i-1}\}$ for all $i$, so $w$ is an RGF. 

Before continuing, it will be useful to see how, given a down step of $P$, we can find its paired up step from $w=\psi(P)$.  If $s_i=D$ is paired with $s_j=U$, $j<i$, then we know that $w_{i+1}=w_j$.  We claim that $w_k>w_j$ for all $k$ in the interval $[j+1,i]$.  It follows that $j$ can be characterized as the largest index $j<i$ with $w_j=w_{i+1}$.  To prove the claim, consider $s_{k-1}$.  If $s_{k-1}=U$ or $b$ then $w_k$ is a left to right maximum and so the desired inequlaity holds.  If $s_{k-1}=a$ or $D$ then $w_k$ equals an entry whose index is earlier in the interval $[j+1,i]$ and so we are again done by induction on $k$.

To show $w=\psi(P)\in R_n(1212)$ suppose, towards a contradiction, that $w$ contains the pattern $1212$  so that we have a subword $w_iw_jw_kw_l$ with $w_i=w_k=x$ and $w_j=w_l=y$.  Pick $i$ to be the largest index $i<j$ such that $w_i=x$, and $k$ to be the smallest index $k>j$ such that $w_j=x$.  Thus there are no copies of $x$ between $w_i$ and $w_k$.  Once $i,k$ are chosen, do the same with $j,l$ so that there is no copy of $y$ between $w_j$ and $w_l$.  Now $w_k$ is not a left to right maximum since it is  preceded by $w_i=w_k$, and $w_k\neq w_{k-1}$ since there is no copy of $x$ between $w_i,w_k$ and $i<j<k$.  It follows from the definition of $\psi$ and the choice of $i,k$ that $s_{k-1}=D$ and is paired with $s_i=U$.  Similarly we get that $s_{l-1}=D$ and is paired with $s_j=U$.  It follows that $P$ has the subword $s_i s_j s_{k-1} s_{l-1} =UUDD$ where the first $U$ and first $D$ are paired, and the second $U$ and second $D$ are paired.  But this kind of pairing can not happen in a Motzkin path and so we have our desired contradiction.

To motivate the definition of the inverse note that, in the definition of $\psi$,  if $s_i=U$ then we have an increase $w_{i}<w_{i+1}$. Since the up step must have a paired down step $s_j$ there must be some $j>i$ with  $w_j=w_i$.   If instead $s_i=b$ we have an increase $w_{i}<w_{i+1}$, but our map will not further repeat $w_i$.
If $s_i=a$ then $w_{i}=w_{i+1}$. 
Finally, if $s_i=D$ then it follows from the discussion two paragraphs before that $w_i>w_{i+1}$.
This  leads us to define, 
for $w\in R_n(1212)$ , the lattice path $\psi^{-1}(w)=P=s_1\dots s_{n-1}$ where
$$
s_i = \left\{
\begin{array}{l}
a \text{ if  $w_{i}=w_{i+1}$},\\
b \text{ if $w_{i}<w_{i+1}$ and there does not exist $j>i+1$ such that $w_j=w_{i}$},\\
U \text{ if $w_{i}<w_{i+1}$ and there exists $j>i+1$ such that $w_j=w_{i}$},\\
D \text{ if $w_{i}>w_{i+1}$}.\\
\end{array}
\right.
$$
By our previous discussion, this map is an inverse on the image of $\psi$.  Since it is known that $|\cM_{n-1}^2|=C_n=|R_n(1212)|$, 
where $C_n$ is the $n$th Catalan number, $\psi$ must be a bijection. 

Lastly we will show that $\area(P)=\rs(\psi(P))$. 
Consider a letter $w_i$. We want to count the number of distinct elements to the  right and smaller than $w_i$. 
We will do so by considering the first occurrence of such an element to contribute to $\rs$, while all other copies of the same element do not.
We must find which steps $s_k$ with $k\geq i$ make $w_{k+1}$ smaller than $w_i$. 
If $s_k=a$ then $w_{k+1}=w_k$  and so $w_{k+1}$ is not a first occurrence.
 If $s_k$ equals $b$ or $U$ then $w_{k+1}$ is a left to right maximum and so  not smaller than $w_i$. 
So the only steps which could result in something to the  right and smaller than $w_i$ are down steps $s_k=D$. Let $s_{\ell}=U$ be its paired up step. First we will consider the case when $\ell=i$. In this case, $w_{k+1}=w_i$ so $w_{k+1}$ is not smaller than $w_i$. If instead $\ell>i$, we have $w_{k+1}=w_{\ell}$ and $w_{k+1}$ is not a first occurrence. Our last case is that $\ell <i$.   But then $i\in [\ell+1,k]$ which, as we proved earlier, implies $w_i> w_{k+1}$
This shows that $w_{k+1}$ is an element to the right and smaller than $w_i$.  Finally, we also have that for all $j$ in $[i,k]$, $w_j>w_{k+1}$. Thus $w_{k+1}$ is the first occurrence of this letter which appears to the right of $w_i$, and so $w_{k+1}$ is counted by $\rs$.

This means that $\rs(w_i)$ is equal to the number of down steps weakly to the right of $s_i$ such that its paired up step is strictly to the left of $s_i$. In the case of $s_i$ equal to $a$,  $b$, or $U$  this calculation is equal to the  level of the step. In the case of $s_i=D$  this calculation is equal to level of the step plus one. Adding the contributions from all the $s_i$ gives  the total area under the path $P$. Since this also counts  $\rs(w)$ we have that $\area(P)=\rs(w)$. 
\eprf

\subsection{Combinations with other patterns}

Next we examine RGFs that avoid $1212$ and another pattern of length $3$. As the patterns $121$, $122$, and $112$ are all subpatterns of $1212$, the only interesting cases to look at are $R_n(111,1212)$ and $R_n(123,1212)$. We start by calculating $\RS_n(111,1212)$. It is easy to combine Theorem~\ref{rgf:len3} and Lemma~\ref{noncross_avoid} to characterize $R_n(111,1212)$. 
\ble
\label{avoid_111_1212}
We have 
$$
R_n(111,1212)=\{w\in R_n(1212)\ : \text{ every element of } w \text{ appears at most twice}\}. 
$$
for all $n\ge0$.\hqed
\ele

The following proposition is similar to Theorem~\ref{rs1212} in many respects. First, this proposition provides a $q$-analogue of the standard Motzkin recursion and is proved using techniques similar to those used previously. Furthermore, it will also be used to connect $R_n(111,1212)$ to lattice paths.

\bpr
\label{RS(111,1212)}
We have 
\begin{align*}
\RS_0(111,1212)&=1,\\
\RS_1(111,1212)&=1,
\end{align*}
and for $n\geq 2$,
$$
\RS_n(111,1212)=\RS_{n-1}(111,1212)+\sum_{k=0}^{n-2}q^k\RS_k(111,1212)\RS_{n-k-2}(111,1212).
$$ 
\epr
\bprf
We follow the proof of Theorem~\ref{rs1212} by partitioning $R_n(111,1212)$ into the sets 
\begin{align*}
X&=\{w\in R_n(111,1212)\ :\ w_1=1 \text{ and there are no other 1s in }w\},\\
Y&=\{w\in R_n(111,1212)\ :\ w_1w_2=11\},\\
Z&=\{w\in R_n(111,1212)\ :\ w_1w_2=12 \text{ and there is a single other 1 in } w\}.
\end{align*}
Using the same reasoning as in Theorem~\ref{rs1212} and adding the restrictions of avoiding $111$ gives the equivalent characterizations 
\begin{align*}
X&=\{w=1(u+1)\ :\ u\in R_{n-1}(111,1212)\},\\
Y&=\{w=11(u+1)\ :\ u\in R_{n-2}(111,1212)\},\\
Z&=\{w=1(u+1)1v\ :\ u\in R_k(111, 1212)\text{ for }1\leq k\leq n-2,\\
&\hspace{120pt}\st(v)\in R_{n-k-2}(111, 1212),v\cap 1(u+1)=\emptyset\}.
\end{align*}
From this, the desired recurrence easily follows. 
\eprf

The next result provides an explicit bijection between $R_n(111,1212)$ and $\cM_n$. We first extend  the level statistic defined in the previous subsection to paths. Given a Motzkin path $P=s_1\dots s_n$, we define the level of the path, $l(P)$, to be 
$$
l(P)=\sum_{i=1}^nl(s_i).
$$
In Figure~\ref{tcm}, $l(P)=10$.
It should be noted that if we impose a rectangular grid of unit squares on the first quadrant of the plane, then $l(P)$ simply counts the total area of the unit squares contained below $P$ and above the $x$-axis. We will use our bijection to calculate the generating function for the level statistic taken over all Motzkin paths of length $n$.

\bth
For $n\geq 0$, we have 
$$
\RS_n(111,1212)=\sum_{P\in\cM_n}q^{l(P)}.
$$
\eth
\bprf
We start by defining a bijection $\phi:R_n(111,1212)\mapsto \cM_n$. For any $w=w_1\dots w_n$, we let $\phi(w)=P$, where $P=s_1\dots s_n$ and 
$$
s_i=\begin{cases}
U\text{ if }w_i=w_j\text{ for some }j>i,\\
H\text{ if }w_i\neq w_j\text{ for any }j\neq i,\\
D\text{ if }w_i=w_j\text{ for some }{j<i}.
\end{cases}
$$
To show that $\phi$  is well defined, first note that since $w$ contains at most two copies of any integer, the three cases are disjoint and cover all possibilities.  We also need to show that $P$ is a Motzkin path.  But this is true because the definition of $\phi$ induces a bijection between the up steps and down steps of $\phi(w)$ in which each up step precedes its corresponding down step. 

We will need the fact that  this bijection between up and down steps induced by the definition of $\phi$ is exactly the same as the pairing relationship in the path $\phi(w)$. Formally, we have that $i<j$ and $w_i=w_j$ if and only if $s_i$ is the up step paired with the down step $s_j$. To see this, assume $i<j$ and $w_i=w_j$. Consider the subword $w_i\dots w_j$. As $w$ avoids $111$ and $1212$, we must have $w_k>w_i$ for each $i<k<j$. 
Furthermore, if $i<k<j$ and if $w_k=w_{k'}$ for some other $k'$, we must also have $i<k'<j$ since $w$ has no $xyxy$ pattern. Thus the subpath $s_{i+1}\dots s_{j-1}$ is a Motzkin path translated to start at the level of $s_{i+1}$. It follows that $s_i$ and $s_j$ must be paired. This in fact proves the equivalence, as the pairing relationship on a Motzkin path is unique. 

To invert $\phi$, note first that for  $w\in R_n(111,1212)$, the left to right maxima occur precisely at those $i$ corresponding to the first two cases in the definition of $\phi$.
So given $P=s_1\dots s_n$  a path in $\cM_n$, we define $\phi^{-1}(P)=w_1\dots w_n$ by $w_1=1$ and, 
for $j\geq2$,
$$
w_j=
\begin{cases}
\begin{array}{ll}
1+\max\{w_1,\dots, w_{j-1}\} &\text{ if }s_j=U\text{ or }s_j=H,\\
w_i &\text{ if } s_j \text{ is a down step paired with } s_i.
\end{array}
\end{cases}
$$
The proof that this function is well defined is similar to the one given for $\phi$ and so omitted.
And from the description of $\phi$ in terms of left to right maxima as well as our remarks about $\phi$'s relationship to the pairing bijection, it should be clear that this is the inverse function.

It now suffices to show that $\rs(w)=l(\phi(w))$ for any $w$ in our avoidance class. 
Let $w=w_1\dots w_n$ and $\phi(w)=s_1\dots s_n$. We will prove the stronger statement that $\rs(w_i)=l(s_i)$ for $1\leq i\leq n$. To do this, note that if $l(s_i)=k$, then there are precisely $k$ down steps to the right of $s_i$ whose paired up steps precede $s_i$ in $\phi(w)$. 

Now assume $\rs(w_i)=k$. By definition, there are $k$ integers $w_{j_1},\dots,w_{j_k}$ to the right of and smaller than $w_i$. As $w$ is an RGF, each of these integers also appear to the left of $w_i$ in $w$. By the definition of $\phi$, the $s_{j_1},\dots,s_{j_k}$ are down steps which follow $s_i$ in $\phi(w)$ whose paired up steps precede $s_i$. This gives $l(s_i)\geq k$.

To see that we actually have equality, assume that there is another down step $s_l$ which follows $s_i$ in $\phi(w)$. We know that in $w$, $w_i\leq w_l$, as $w_l$ does not contribute to $\rs(w_i)$. If $w_i=w_l$, then in fact $s_i$ and $s_l$ must be paired via level, and thus $s_l$ does not change $l(s_i)$. Finally, we deal with the case $w_i<w_l$. As $s_l$ is a down step, there must exist another letter $w_{l'}$ in $w$ with $l'<l$ and $w_{l'}=w_l$. In order for $w$ to be an RGF and to avoid $1212$, one can see that we must also have $i<l'$. 
Hence $s_l$ and its paired up step $s_{l'}$ both follow $s_i$ in $\phi(w)$, and thus $s_l$ will still not affect $l(s_i)$. Thus we have $l(s_i)=k=\rs(w_i)$ as desired. 
\eprf

We conclude the section with a simple proposition characterizing $R_n(123,1212)$. As the result follows easily from Theorem~\ref{rgf:len3},  Proposition~\ref{noncross_avoid}, and standard counting techniques, we leave the proof to the reader.  
\bpr
\label{R(123,1212)}
If $w$ is contained in $R_n(123,1212)$, then 
$$
w=1^l2^i1^{n-i-l}
$$
for some $l\geq 1$, $i\geq0$ satisfying $l+i\leq n$. As such, for $n\ge0$ we have 
$$
\LB_n(123,1212)=\RS_n(123,1212)=1+\sum_{k=0}^{n-2}(n-k-1)q^k
$$
and

\medskip

\eqed{\LS_n(123,1212)=\RB_n(123,1212)=1+\sum_{k=1}^{n-1}(n-k)q^k.}

\epr


\section{The pattern $1221$}
\label{p1221}

\subsection{Nonnesting partitions}

The term ``nonnesting" has been defined in different ways in the literature.   In some sources a nonnesting partition is a partition $\pi$ where we can never find four elements $a<x<y<b$ such that $a,b\in A$ and $x,y\in B$ for two distinct blocks $A,B$. This is the sense used in Klazar's paper~\cite{kla:aas} and is equivalent to a partition avoiding $14/23$. 

In other papers, including Klazar's article~\cite{kla:cpsII},
 a   partition $\pi$ is nonnesting if, whenever there are four elements $a<x<y<b$ such that $a,b\in A$ and $x,y\in B$ for two distinct blocks $A,B$, then there exists a $c\in A$ such that $x<c<y$. This definition is often given using arc diagrams.  We draw the {\it arc diagram} of a partition of $[n]$ by writing $1$ through $n$ on a straight line and drawing arcs $(a,b)$ if $a<b$ are in a block and consecutive when writing the block in increasing order, see Figure~\ref{arc diagrams}.
A {\it nesting} is a pair of arcs $(a,b)$ and $(x,y)$ such that $a<x<y<b$, and we will say in this case that the pair of arcs {\it nest}.
For completeness, we prove that having no nesting arcs is equivalent to the second definition of a nonnesting partition.
 It is known that the number of partitions satisfying either of these two equivalent conditions is the Catalan numbers, $C_n$.
\begin{prop} The following conditions are equivalent for a partition $\pi$. 
\begin{enumerate}
\item If there are four elements $a<x<y<b$ such that $a,b\in A$ and $x,y\in B$ for two distinct blocks $A,B$, then there exists a $c\in A$ such that $x<c<y$.
\item The arc diagram for $\pi$ contains no nestings. 
\end{enumerate}
\end{prop}
\begin{proof}
We will first show that if a partition fails condition $1$, then its arc diagram has a nesting. Say that there are four elements $a<x<y<b$ such that $a,b\in A$ and $x,y\in B$ for two distinct blocks $A,B$ but there is no $c\in A$ such that $x<c<y$. Since there is an $a\in A$ with $a<x$ there is a largest element $\bar{a}\in A$ where $\bar{a}<x$. Similarly, there is a smallest $\bar{b}\in A$ with $y<\bar{b}$.  Since there is no element of $A$ between $x$ and $y$, $(\bar{a},\bar{b})$ must be an arc.
Also there is  a smallest element $\bar{y}\in B$ such that $x<\bar{y}$ so that $(x,\bar{y})$ is an arc.  Since $\bar{a}<x<\bar{y}<\bar{b}$ these arcs nest which is a contradiction.

Conversely, assume that the arc diagram has two arcs $(a,b)$ and $(x,y)$ which nest with $a<x<y<b$. By construction of the arcs, this implies that $a$ and $b$ are consecutive elements in their block $A$, so there does not exist a $c\in A$ such that $a<c<b$ and  the first condition is false.  
\end{proof}

There is another notion of nonnesting which we will call left nonnesting and can be defined by a different collection of arcs.  For each block $B$ we will draw all arcs of the form $(\min B, b)$ with $b\in B\setminus\{\min B\}$, and call the diagram with these arcs the {\it left arc diagram}. 
An example is displayed in Figure~\ref{arc diagrams}.
If a partition's left arc diagram has no pair of arcs which nest then we will call this partition  {\it left nonnesting} 
to distinguish our term from the previous two definitions of nonnesting. Let this set be $\text{LNN}_n$. 


\begin{prop} We have
$$R_n(1221)=w(\text{LNN}_n).$$
\end{prop}

\begin{proof} First we will show that if a partition's left arc diagram contains a nesting then its associated RGF has the pattern $1221$. 
Let $\pi=B_1/\dots /B_k$ be a partition of $[n]$. Say that its left arc diagram has a nesting which means that we have arcs $(a,b)$ and $(x,y)$ such that $a<x<y<b$. 
Since these are arcs from the left arc diagram we know that $a=\min B_i$ and $x=\min B_j$ for some distinct blocks $B_i$ and $B_j$,
 and since we order the blocks of $\pi$ so that their minimum elements increase we know that $i<j$. As result $w(\pi)$ has the subword $ijji$ which is the pattern $1221$. 

Conversely, say that we have an RGF $w$ with the pattern $1221$, so it has a subword $ijji$ with $i<j$.  Pick the subword so that the first $i$ and $j$ are the first occurrences of these letters in $w$.  Thus they correspond to minima in their respective blocks of the corresponding partition $\pi$.  It folllows that  the two $i$'s and two $j$'s  give rise to nesting arcs in the left arc diagram of $\pi$.
\end{proof}

 \begin{figure}
  \begin{center}
   \begin{tikzpicture}[scale = .8]
\filldraw [black] 
(1,0) circle (2pt)
(2,0) circle (2pt)
(3,0) circle (2pt)
(4,0) circle (2pt)
(5,0) circle (2pt)
(6,0) circle (2pt)
(7,0) circle (2pt);
\draw (1,-.5) node {$1$};
\draw (2,-.5) node {$2$};
\draw (3,-.5) node {$3$};
\draw (4,-.5) node {$4$};
\draw (5,-.5) node {$5$};
\draw (6,-.5) node {$6$};
\draw (7,-.5) node {$7$};
\draw (1,0) .. controls (1,1) and (3,1) .. (3,0);
\draw (3,0) .. controls (3,.7) and (4,.7) .. (4,0);
\draw (2,0) .. controls (2,1.5) and (6,1.5) .. (6,0);
\draw (6,0) .. controls (6,.7) and (7,.7) .. (7,0);
 \end{tikzpicture} 
  \hspace{.5in}
\begin{tikzpicture}[scale = .8]
\filldraw [black] 
(1,0) circle (2pt)
(2,0) circle (2pt)
(3,0) circle (2pt)
(4,0) circle (2pt)
(5,0) circle (2pt)
(6,0) circle (2pt)
(7,0) circle (2pt);
\draw (1,-.5) node {$1$};
\draw (2,-.5) node {$2$};
\draw (3,-.5) node {$3$};
\draw (4,-.5) node {$4$};
\draw (5,-.5) node {$5$};
\draw (6,-.5) node {$6$};
\draw (7,-.5) node {$7$};
\draw (1,0) .. controls (1,1) and (3,1) .. (3,0);
\draw (1,0) .. controls (1,1.5) and (4,1.5) .. (4,0);
\draw (2,0) .. controls (2,1) and (6,1) .. (6,0);
\draw (2,0) .. controls (2,1.5) and (7,1.5) .. (7,0);
 \end{tikzpicture} 
 \caption{The  arc diagram and left arc diagram for the partition $134/267/5$.}
  \label{arc diagrams}
   \end{center}
 \end{figure}
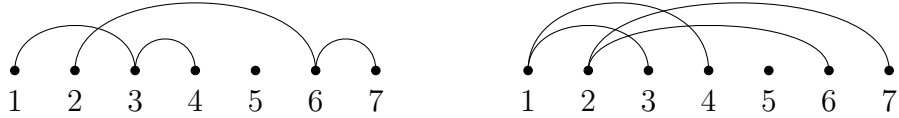

The rest of this section will describe $R_n(1221)$, some of its generating functions, and some connections to other patterns. We will prove that $\LB_n(1221)=\RS_n(1212)$ by showing that there exists a bijection from two-colored Motzkin paths to $R_n(1221)$ which maps area  to $\lb$, and then the result will follow from Theorem~\ref{1212_with_2cmotz}. We further use this bijection and  previous methods to determine the generating function for some  pairs of RGFs which include $1221$. We end the section by showing $\LB_n(1221)=\LB_n(1212)$ and summarizing all the equalities we have proved.

\subsection{{The pattern $1221$ by itself}}

For an RGF $w=w_1 \dots w_n$ we will call a letter $w_i$ {\it{repeated}} if there exists a $j<i$ such that $w_j=w_i$. If a letter is not a repeated letter, we will call it a {\it{first occurrence}}.  Since $w$ is an RGF, the first occurrences are exactly the left to right maxima.

\begin{lem}
A word $w\in R_n(1221)$ if and only if the subword of all repeated letters in $w$ is weakly increasing. 
\label{1221structure}
\end{lem}

\begin{proof}
Say that $w$ contains the pattern $1221$ and so has a subword $xyyx$ for some $x<y$. 
The second $yx$ are repeated letters in $w$. This implies that there is a decrease in the subword of all repeated letters. 

Conversely, say that the subword of all repeated letters of $w$ has an decrease $yx$ with $x<y$.
 Since these are repeated letters in an RGF the first $y$ of $w$ appears earlier, and the first $x$ in $w$ appears earlier than the first $y$. Hence we have a subword $xyyx$ with $x<y$ and the pattern $1221$. 
\end{proof}

Using the previous lemma we can define a surjection $\inc:R_n\rightarrow R_n(1221)$. The map will take a $w\in R_n$ and will output $\inc(w)=v$ which is $w$ with its subword of repeated letters put in weakly increasing order. For example if $w=1112221331$ then $\inc(w)=1112112323$. 

To see this map is well defined we must first show that $v$ is an RGF. But the subword of repeated letters is rearranged to be weakly increasing which forces the maximum of any prefix to weakly decrease.  Since the left to right maxima of $w$  do not move in this process, they do not change in passing  to $v$ so that the latter is still an RGF.  Also, $v$ avoids $1221$ by
Lemma~\ref{1221structure}, showing $\inc$ is well defined.

In the next lemma we show that $\inc$ preserves $\lb$. Note that because $w$ is an RGF, all the numbers in the interval $[w_i+1,\max\{w_1,\dots,w_{i-1}\}]$ appear to the left of $w_i$ and are larger than $w_i$, so 
\begin{equation}
\lb(w_i)=\max\{w_1,\dots,w_{i-1}\}-w_i.
\label{eq:lb(w_i)}
\end{equation}

\begin{lem}
Let $v$ be a rearrangement of $w$ such that both have the same left to right maxima in the same places.
Then $\lb(v)=\lb(w)$. In particular, $\lb(w)=\lb(\inc(w))$. 
\label{lbPreserved}
\end{lem}

\begin{proof}  
Since $w$ and $v$ only have their repeated letters rearranged and their left to right maxima fixed,
we know $\max\{w_1,\dots,w_{i}\}=\max\{v_1,\dots,v_{i}\}$ for all $i$ and $\{v_1,\dots ,v_n\}=\{w_1,\dots ,w_n\} $ as multisets. Using Equation~\ref{eq:lb(w_i)},
$$\lb(w)=\sum_{i=1}^n(\max\{w_1,\dots,w_{i-1}\}-w_i)=\sum_{i=1}^n(\max\{v_1,\dots,v_{i-1}\}-v_i)=\lb(v).$$
The special case of $v=\inc(w)$ now follows from the definition of the function.
\end{proof}

We wish to show that the generating function $\RS_n(1212)$ discussed in Section~\ref{p1212}
is equal to $\LB_n(1221)$.  The proof will be similar to that of Theorem~\ref{1212_with_2cmotz} in that we will construct a bijection $\beta$  from two-colored Motzkin paths length $n-1$ to $R_n(1221)$ which maps  $\area$ to $\lb$. 
The map $\beta$ will not be difficult to describe. However, proving that $\beta$ is a bijection will require a detailed argument. We define a map $\alpha:R_k(1221)\rightarrow R_{k+2}(1221)$ and provide the following lemma to assist us. 
This map will be useful when discussing two-colored Motzkin paths which are obtained from a smaller path by prepending an up step and appending a down step.
Given any $v\in R_k(1221)$ we define $\bar{v}=\bar{v}_1\bar{v}_2\dots \bar{v}_k$ such that
\beq
\label{bar{v}}
\bar{v}_i=\begin{cases}
\begin{array}{ll}
v_i+1&\text{ if } v_i \text{ is a first occurrence},\\
v_i&\text{ else}.
\end{array}
\end{cases}
\eeq
It is not hard to see that $u=1\bar{v}1$ is an RGF, but it may not avoid $1221$, so we define 
$$\alpha(v)=\inc(u)$$ which is in $R_{k+2}(1221)$ by Lemma~\ref{1221structure}. 
For example, if $v=1212344$ 
will have  $u=1\bar{v}1=123124541$ and $\alpha(v)=123114524$.

\begin{lem}
\label{alpha map}
For $k\geq 0$ the map $\alpha:R_k(1221)\rightarrow R_{k+2}(1221)$ is an injection. Furthermore, the image of $\alpha$ is precisely the $w\in  R_{k+2}(1221)$ satisfying the following three properties. 
\begin{enumerate}
\item[(i)] The word $w$ has more than one $1$ and ends in a repeated letter.
\item[(ii)] If $w_i$ is a repeated letter then $w_i<\max\{w_1,\dots ,w_{i-1}\}$.
\item[(iii)] If, for $i\le j$, we have $w_{i-1}$ and $w_{j+1}$ are repeated letters with $w_i w_{i+1}\dots w_j$  all first  occurrences 
 then  $w_{j+1}<w_i-1$.
\end{enumerate}

\end{lem}
\begin{proof} 
We will start by showing that $\alpha$ is injective. Given a $v\in R_k(1221)$, consider $u=1\bar{v}1$. We can easily recover $\bar{v}$ from $u$ by removing the first and last $1$, and can further recover $v$ by decreasing 
all left to right maxima in
 $\bar{v}$ by one. We finish showing that $\alpha$ is injective by recovering $u$ from $w=\inc(u)$. Note that since $v$ avoids $1221$, its subword $r$ of all repeated letters is weakly increasing. The subword of all repeated letters in $u=1\bar{v}1$ is then $r1$.  Making this subword increasing results in the subword of all repeated letters in $w$ being $1r$. We can thus recover $u$ by replacing $1r$ in $w$ by $r1$. 

Next, we show that $w$ satisfies all three properties. Since $u=1\bar{v}1$ has more than one $1$ and ends in a repeated letter, the RGF $w=\inc(u)$ does as well. Property (i) is thus satisfied. Next we show property (ii) by first showing that $u$ satisfies property (ii). 
If $v_i$ is a repeated letter then we always have $v_i\leq \max\{v_1,\dots ,v_{i-1}\}$. Since we increased all first occurrences to get $\bar{v}$  and left the repeated letters the same we have $ \bar{v}_i< \max\{\bar{v}_1,\dots ,\bar{v}_{i-1}\}$. And clearly the two new ones in $u$ do not change this inequality.
As previously noted, the value in the place of a given repeated letter can only get weakly smaller in passing from $u$ to $w=\inc(u)$.
And since left to right maxima don't change, $w$ also satisfies  property (ii). Lastly, we will show property (iii). 
Consider the situation where $w_i w_{i+1}\dots w_j$  are all  first occurrences but $w_{i-1}$ and $w_{j+1}$ are  repeated letters. 
But then $w_{j+1}$ was in position $i-1$ in $u$ which is also a position in $\bar{v}$.  And the element in position $i$ of $u$ is $w_i$ which is a left to right maximum.  Since left to right maxima in $v$ were increased by one in passing from $v$ to $\bar{v}$ we have $w_{j+1}<w_i-1$ as desired.
\end{proof}

Our goal is to define a map $\beta:\cM^2_{n-1}\rightarrow R_n(1221)$ which maps $\area$ to $\lb$. Before we define $\beta$ we will discuss a partition of the region under $R=s_1 \dots s_{n-1}\in \cM^2_{n-1}$ which will aid us in this task. Figure~\ref{path} gives an example of this process where different shadings indicate parts of the partition.
Recall that $l(s_i)$ is the level, 
or smallest $y$-value, of $s_i$. If $s_i=D$, we define $A(s_i)$ to be equal to the area in the same row between $s_i$ and its paired up step but excluding the area under other down steps or $a$-steps. 
In Figure~\ref{path}, $A(s_5)=1$, $A(s_8)=A(s_{12})=2$ and $A(s_9)=5$.
The area under $R$ can be partitioned as follows. 
The  rectangle under an $a$-step  $s_i$ will be a part with  area $l(s_i)$. For example, in the figure we have the area $l(s_4)=2$.
Our other parts will be associated to down steps. Given a down step $s_i$, its part will consist of the region counted by $A(s_i)$ together with the rectangle of squares under the down step whose area  is given by $l(s_i)$, for a total area of $A(s_i)+l(s_i)$.  Returning to our example,  steps $s_5, s_8, s_9,$ and $s_{12}$ contribute total areas $2,3,5,$ and $2$ (respectively).
 Since this partitions all the region under $R$ we have 
\begin{equation}
\area(R)=\sum_{s_i=a}l(s_i)+\sum_{s_i=D}(A(s_i)+l(s_i)).
\label{eq:area(R)}
\end{equation}

Next we will define a map $\beta:\cM^2_{n-1}\rightarrow R_n(1221)$ such that $\area(R)=\lb(\beta(R))$. Before we define $\beta(R)$ we will define an RGF, $v(R)=v_1\dots v_n$, by letting $v_1=1$ and
$$
v_{i+1} = \left\{
\begin{array}{ll}
\max\{v_1,\dots, v_{i}\}+1 &\text{ if  $s_i$ equals $U$ or $b$},\\
\max\{v_1,\dots, v_{i}\}-l(s_i)& \text{ if $s_i=a$},\\
\max\{v_1,\dots, v_{i}\}-A(s_i)-l(s_i)& \text{ if $s_i=D$},
\end{array}
\right.
$$
for $i\ge0$.  For the two-colored Motzkin path $R$ in Figure~\ref{path} we have $v(R)=1234225631786$.

A comparison of the first case in the definition of $v$ with the other two shows that the left to right maxima of $v$ are consecutive integers starting at $1$.  So to show  that $v$ is an RGF we only have to prove that $v_{i+1}>0$ for all $s_i\in \{a,D\}$. Note that for all $i\geq 1$ we have that $\max\{v_1,\dots,v_i\}$ is equal to one more than the number of $b$-steps plus the number of up steps in the first $i-1$ steps. The level $l(s_i)$ of any horizontal step is at most the number of previous up steps, so for $s_i=a$ we have $v_{i+1}=\max\{v_1,\dots, v_{i}\}-l(s_i)>0$. Note that the area counted by $A(s_i)$ between $s_i=D$ and its corresponding up step excluding the area under other $a$-steps or down steps is at most the number of up steps plus $b$-steps between and including the paired up and down step. Also, the level of the down step is at most the number of up steps strictly before its paired up step. All together $A(s_i)+l(s_i)$ is at most the number of up steps and $b$-steps in the first $i-1$ steps. As result, for $s_i=D$  we have $v_{i+1}=\max\{v_1,\dots, v_{i}\}-A(s_i)-l(s_i)>0$. Hence, $v$ is an RGF. However, $v(R)$ may not avoid $1221$, so we define 
$$\beta(R)=\inc(v(R))$$
which avoids $1221$ by Lemma~\ref{1221structure}.
 For the two-colored Motzkin path $R$ in Figure~\ref{path} we have $\beta(R)=1234125623786$.

Next we show that $\area(R)=\lb(v)$ which will imply that $\area(R)=\lb(\beta(R))$ by Lemma~\ref{lbPreserved}. 
It is easy to see that $\lb(v_1)=0$ and if $s_i$ is $b$ or $U$ then $\lb(v_{i+1})=0$. 
Next consider $s_i=a$ so $v_{i+1}=\max\{v_1,\dots, v_{i}\}-l(s_i)$. By Equation~(\ref{eq:lb(w_i)}), we have $\lb(v_{i+1})=l(s_i)$.
Lastly, if $s_i=D$ then $v_{i+1}=\max\{v_1,\dots, v_{i}\}-A(s_i)-l(s_i)$. By Equation~(\ref{eq:lb(w_i)}) again, $\lb(v_{i+1})=A(s_i)+l(s_i)$. As a 
result 
$$\lb(v)=\sum_{s_i=a}l(s_i)+\sum_{s_i=D}(A(s_i)+l(s_i))=\area(R)$$
 by equation~\ref{eq:area(R)}.

 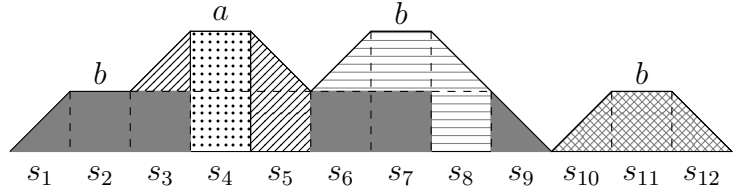
\begin{figure}
  \begin{center}
\begin{tikzpicture}[scale = .8]
\filldraw[pattern =  dots] (3,0) --  (4,0) -- (4,2) -- (3,2) -- (3,0);
\filldraw[pattern = north east lines] (2,1)--(3,1)--(3,2)--(2,1);
\filldraw[pattern = north east lines] (4,0)--(5,0)--(5,1)--(4,2)--(4,0);
\filldraw[pattern =  horizontal lines, pattern color = gray] (5,1)--(7,1)--(7,0)--(8,0)--(8,1)--(7,2)--(6,2)--(5,1);
\filldraw[gray] (0,0)--(3,0)--(3,1)--(1,1)--(0,0);
\filldraw[gray] (5,0)--(7,0)--(7,1)--(5,1)--(5,0);
\filldraw[gray](8,0)--(9,0)--(8,1)--(8,0);
\filldraw[pattern = crosshatch, pattern color = gray] (9,0)--(12,0)--(11,1)--(10,1)--(9,0);
\draw (0,0) --  (1,1) -- (2,1) -- (3,2) -- (4,2)--(5,1) --(6,2) --(7,2) --(8,1) --(9,0) --(10,1) --(11,1)--(12,0)  ;
\draw [dashed]  (2,1) -- (8,1);
\draw [dashed]  (1,1) -- (1,0);
\draw [dashed]  (2,1) -- (2,0);
\draw [dashed]  (3,2) -- (3,0);
\draw [dashed]  (4,2) -- (4,0);
\draw [dashed]  (5,1) -- (5,0);
\draw [dashed]  (6,2) -- (6,0);
\draw [dashed]  (7,2) -- (7,0);
\draw [dashed]  (8,1) -- (8,0);
\draw [dashed]  (10,1) -- (10,0);
\draw [dashed]  (11,1) -- (11,0);
\draw (1.5,1.3) node {$b$};
\draw (3.5,2.3) node {$a$};
\draw (6.5,2.3) node {$b$};
\draw (10.5,1.3) node {$b$};
\draw(0.5,-.4) node{$s_1$};
\draw(1.5,-.4) node{$s_2$};
\draw(2.5,-.4) node{$s_3$};
\draw(3.5,-.4) node{$s_4$};
\draw(4.5,-.4) node{$s_5$};
\draw(5.5,-.4) node{$s_6$};
\draw(6.5,-.4) node{$s_7$};
\draw(7.5,-.4) node{$s_8$};
\draw(8.5,-.4) node{$s_9$};
\draw(9.5,-.4) node{$s_{10}$};
\draw(10.5,-.4) node{$s_{11}$};
\draw(11.5,-.4) node{$s_{12}$};
 \end{tikzpicture} 
 \caption{The area decomposition of a two-colored Motzkin path}
  \label{path}
   \end{center}
 \end{figure}

We now show that the $\beta$ map behaves nicely with respect to two of the usual decompositions of Motzkin paths.

\begin{lem} Let $P$ and $Q$ be two-colored Motzkin paths with  $\beta(P)=x$ and $\beta(Q)=1y$. The map $\beta$ has the following properties. 
\begin{enumerate}
\item[(1)] $\beta(PQ)=x(y+\max(x)-1)$. 
\item[(2)]  $\beta(UPD)=\alpha(x)$. 
\end{enumerate}
\label{beta map properties}
\end{lem}
\begin{proof} 
To prove statement (1), we first claim that  
$$v(PQ)=v(P)(q+\max(v(P))-1)$$
where $q$ is $v(Q)$ with its initial $1$ deleted.  
It is clear from the definition of $v$ that the first $|P|+1$ positions of $v(PQ)$ are $v(P)$.
Also by definition of $v$, the  first occurrences other than the initial $1$ are in bijection with the union of the up steps and $b$-steps.  It follows that the subword of first occurrences in the  last $|Q|$ positions of $v(PQ)$ is the same as the corresponding subword in $q$ with all elements increased by $\max(v(P)) -1$.  Thus the maximum value in any prefix of $v(PQ)$ ending in these positions is increased over the the corresponding maximum in $q$ by this amount.  Furthermore, the areas and levels of down steps and $a$-steps in $Q$  in that portion of $PQ$ are the same since $P$ ends on the $x$-axis.  So, using the definition of $v$ for these types of steps, the last $|Q|$ positions of $v(PQ)$ are exactly $q'=q+\max(v(P)) -1$.  To prove the equation for $\beta$, it suffices to show that the $\inc$ operator only permutes elements within $v(P)$ and within $q'$.  But this is true because all elements of $q'$ are greater than or equal to those of $v(P)$. 

To prove the second statement, first consider $v:=v(P)=v_1\dots v_k$ and $u:=v(UPD)=u_1\dots u_{k+2}$.  We claim that $u=1\bar{v}1$. Clearly $u$ begins with a $1$.  To see it must also end with $1$, note that since the last step of 
$UPD=s_1\dots s_{k+1}$ is down step and this path does not touch the axis between its initial and final points, we have $l(s_{k+1})=0$ and $A(s_{k+1})$ is the total number of up steps and $b$-steps in $UPD$.  It now follows from the definition of the map $v$ and our interpretation of the maximum of a prefix that $u_{k+2}=1$.  Let $u'$ be $u$ with its initial and final $1$'s removed.  To see that $u'=\bar{v}$, first note that every step of $UPD$ except the first is preceded by one more up step than in $P$.  It follows every first occurrence  of $v$ is increased by one in passing to $u'$.  But the area under each $a$-step and under each down step also increases by one during that passage.  So the differences defining the $v$-map in such cases will stay the same for these repeated entries.  It should now be clear that $u'=\bar{v}$.  It follows immediately that $\beta(UPD)=\inc(1\bar{v}1)=\alpha(x)$. 
\end{proof}

Before we show that $\beta$  is a bijection, we will 
need a method for determining from the image of a path where that path first returns to the $x$-axis.  The following lemma will provide the key.

\begin{lem}
Given paths $P\in \cM_{k-3}^2$ and $Q$ with $k\geq 3$, the word $\beta(UPDQ)=w$ has $w_k$ as the right-most repeated letter such that $w_1\dots w_k$ satisfies all three properties in  Lemma~\ref{alpha map}.
\label{beta map properties 2}
\end{lem}
\begin{proof}
Given a path $R=UPDQ\in \cM_{n-1}^2$ as stated, by Lemma~\ref{beta map properties} we know that if we write $\beta(Q)=1q$ then 
\beq
\label{decomp}
w=\beta(R)=\alpha(\beta(P))(q+m-1)
\eeq
where $m=\max(\alpha(\beta(P)))$.
Lemma~\ref{alpha map} implies that the prefix $w_1\dots w_k=\alpha(\beta(P))$ satisfies all three properties.
So it suffices to show that if there exists another repeated letter $w_i$ after $w_k$ then $w_1\dots w_i$ fails propertry (ii) or property (iii). In particular,  it suffices to show such a failure for the prefix where $w_i$ is the next repeated letter after $w_k$ since any other prefix under consideration contains $w_1\dots w_i$.  

If $i=k+1$ then, since every element of $q$ is increased by $m-1$ and $w_{k+1}$ is repeated, we must have $w_{k+1}=m=\max\{w_1,\dots,w_k\}$, contradicting property (ii).
If instead $i>k+1$ then $w_{k+1}$ is a first occurrence and $w_{k+1}=\max\{w_1,\dots,w_k\}+1=m+1$.  By definition of $w_i$, we have that $w_{k+1},\dots, w_{i-1}$ are all first occurrences with $w_k$ and $w_i$ repeated letters.  Note that all elements in $q$ were at least $1$ and then  increased by $m-1$, so we must have $w_{i}\geq m=w_{k+1}-1$ which contradicts property (iii). 
\end{proof}

It will be helpful for us to be able to refer to the special repeated letter mentioned in the lemma above.  So, given an RGF $w=w_1\dots w_n$, if there exists a  right-most repeated letter  $w_k$  such that $w_1\dots w_k$ satisfies all three properties in Lemma~\ref{alpha map} then we will say that  $w_k$ {\it breaks the word} $w$.  Note that if such a repeated letter exists, its index $k$ is unique.

\begin{thm}
The map $\beta:\cM^2_{n-1}\rightarrow R_n(1221)$ is a bijection and $\area(R)=\lb(\beta(R))$.
\label{2colorMotzkin_R_n(1221)}
\end{thm}
\begin{proof} We have already shown that $\beta$ is a well-defined map and that $\area(R)=\lb(\beta(R))$. 
Since $|\cM^2_{n-1}|=C_n=|R_n(1221)|$, to show $\beta$ is a bijection it suffices to show
 $\beta$ is injective.
We prove this by induction on $n$. It is easy to see that $\beta$ is an injection for $n\leq 2$. We now assume that $n>2$ and $\beta: \cM^2_{k-1}\rightarrow R_k(1221)$ is injective
  for all $k<n$.

We will discuss three cases for paths $R\in \cM_{n-1}^2$ and in each case we will show that  $R$ maps to an RGF distinct from the other RGFs in that case and also from the RGFs in previous cases.

First consider all paths $R$ which start with an $a$-step so that $R=aQ$ for some path $Q$. By Lemma~\ref{beta map properties}, we have $\beta(R)=11y$ where $\beta(Q)=1y$.  Injectivity of the map now follows from the fact that, by induction, it is injective on paths $Q$ of length $n-2$.

Our second case consists of paths $R$ of the form $R=bQ$.  Now $\beta(R)=12(y+1)$ with $y$ as above.  Clearly these are distinct from the words in the previous paragraph and injectivity within this case follows by induction as before.

For the last case, consider all paths $R$ which start with an up step so we can write $R=UPDQ$ for paths $P\in \cM_{k-3}^2$ and $Q$ where $k\geq 3$. By Lemma~\ref{beta map properties} we have equation~(\ref{decomp}),
and by Lemma~\ref{beta map properties 2}  the repeated letter $w_k$  breaks the word $w$. Note that because $\alpha(\beta(P))=w_1\dots w_k$ satisfies property (i) in Lemma~\ref{alpha map},
$w$ has more than one $1$ and so can not agree with a word from the second case above.  But since $R$ starts with an up step, $w$ starts with the prefix $12$ and so can not be a word from the first case.  Finally, by uniqueness of the index of $w_k$, the injectivity of the map $\alpha$, and induction the word $w$ is uniquely determined among all words in this case.  This finishes the proof that $\beta$ is injective.
\end{proof}

Combining the previous result with Corollary~\ref{rs1212=M} and the definition of $M_n(q)$ in~\ref{M_n(q)} we have the following corollary.

\begin{cor} 
\label{LB(1221)=M(q)}
We have

\medskip

\eqed{\LB_n(1221)=\RS_n(1212)=M_{n-1}(q).}
\end{cor}

\subsection{Combinations with other patterns}

Next we consider the RGFs which avoid $1221$ and another length three pattern. Since $121$ and $122$ are subwords of $1221$ these cases are not interesting, so we will focus on $111$, $112$, and $123$. 

\begin{thm} We have for $L_n:=\LB_n(111,1221)$ that $L_0=L_1=1$ and, for $n\ge2$,
and 
$$L_n=L_{n-1}+L_{n-2}+\sum_{k=1}^{n-2}q^{k}L_{k-1}L_{n-k-1}$$
\end{thm}

\begin{proof}
Let $\cN_{n}$ be the collection of two-colored Motzkin paths $R\in \cM_{n}^2$ such that $\beta(R)$ avoids $111$. 
Define $N_{-1}(q)=1$ and, for $n\ge 0$,
$$N_{n}(q)=\sum_{R\in {\cN}_{n}^2}q^{\area(R)}.$$
By Theorem~\ref{2colorMotzkin_R_n(1221)} we only need to  show that $N_n(q)=\LB_{n+1}(111,1221)$ satisfies an equivalent recurrence and initial conditions.  By definition 
$N_{-1}(q)=1$, and $N_0(q)=1$ because of the empty path.  So we wish to show that for $n\ge1$
\beq
\label{N(q)Sum}
N_{n}(q)=N_{n-1}(q)+N_{n-2}(q)+\sum_{k=0}^{n-2}q^{k+1}N_{k-1}(q)N_{n-k-2}(q).
\eeq

We  partition $\cM_{n}^2$ as in the proof of Theorem~\ref{2colorMotzkin_R_n(1221)}:
\begin{align*}
X&=\{R=aQ :\ Q\in \cM^2_{n-1}\},\\
Y&=\{R=bQ :\ Q\in \cM^2_{n-1}\},\\
Z&=\{R=UPDQ :\ P\in \cM^2_{k},\ Q\in \cM^2_{n-k-2}\text{ and } k\in[0,n-2]\}.
\end{align*}
We claim that when we restrict this  partition to paths in $\cN_n$ we have
\begin{align*}
X_\cN&=\{R=abQ :\ Q\in \cN_{n-2}\},\\
Y_\cN&=\{R=bQ :\ Q\in \cN_{n-1}\},\\
Z_1&=\{R=UDQ :\ Q\in \cN_{n-2}\},\\
Z_2&=\{R=UbPDQ :\ P\in \cN_{k-1},\ Q\in \cN_{n-k-2}\text{ and } k\in [n-2]\},
\end{align*}
where the set $Z$ breaks into two subsets. From the second partition
we will be able to deduce  the desired recursion.

Consider a path $R=aQ\in S$. We claim that $\beta(R)$ avoids $111$ if and only if   $Q=bQ'$ for $Q'\in \cN_{n-2}$ which will show that $X$ restricts to $X_\cN$.
If we write $\beta(Q)=1y$ we have $\beta(R)=11y$. The word $\beta(R)$ avoids $111$ if and only if the the word $y$ has no $1$'s and at most two copies of every other number.
Note that the second case considered in Theorem~\ref{2colorMotzkin_R_n(1221)} contained all paths which started with a $b$-step and that these paths were mapped bijectively to words with exactly one $1$.   
It is also clear that $y=\beta(Q')+1$ has at most two copies of every number greater than one if and only if the same is true of $\beta(R)$.  The claim now follows.
Because $\area(R)=\area(Q')$ summing over all paths in this case gives us the term $N_{n-2}(q)$.

If instead $R=bQ\in T$ then, using that notation of  Lemma~\ref{beta map properties}, 
$\beta(R)=12(y+1)=1(\beta(Q)+1)$.  So $\beta(R)$ avoids $111$ if and only if $\beta(Q)$ does.
It follows that $Y$ restricts to $Y_\cN$.
Because $\area(R)=\area(Q)$ summing over all paths in this case gives us the term $N_{n-1}(q)$.

Next, we  consider paths $R=UPDQ$ from the third part, $U$. First consider the case where $P$ has length $0$ so $R=UDQ$. We want to prove that $\beta(R)$ avoids $111$ if and only if $\beta(Q)$ avoids $111$ since this will show that the  collection of paths in $Z$ with $|P|=0$  restricts to $Z_1$. If we write $\beta(Q)=1y$  we have $\beta(R)=121(y+1)$.  Thus $\beta(Q)$ avoids $111$ if and only if $\beta(R)$ does and the restriction is as claimed.
Because $\area(R)=1+\area(Q)$ summing over all paths in this case gives us the term $qN_{n-2}(q)$
which is the $k=0$ term in equation~(\ref{N(q)Sum}).

Lastly, consider a path $R=UPDQ$ with $|P|=k\in[n-2]$ which are the remaining paths in $Z$.  We will show that $\beta(R)$ avoids $111$ if and only if $P=bP'$ and both the words $\beta(P')$ and $\beta(Q)$ avoid $111$. This will show that the remaining paths in $Z$ restrict to $Z_2$ in the second partition. 
First we make an observation about $\alpha(\beta(P))$. Let $m=\max(\beta(P))$ and  $\{1^{s_1},\dots,m^{s_m}\}$ be the multiset of all letters in $\beta(P)$.  The map $\alpha$  increases all first occurrences by one and adds two $1$'s but otherwise doesn't affect the collection of letters. So the multiset of letters in $\alpha(\beta(P))$ is  $\{1^{s_1+1},\dots,m^{s_m},m+1\}$. 
If we write $\beta(Q)=1y$ then  we have $\beta(R)=\alpha(\beta(P))(y+m)$ since $m=\max(\alpha(\beta(P)))-1$. 
If $\{1^{t_1},\dots,\bar{m}^{t_{\bar{m}}}\}$ is the multiset of letters in $\beta(Q)$ then the multiset of letters in $\beta(R)$ is $\{1^{s_1+1},\dots,m^{s_m},(m+1)^{t_1},\dots, (m+\bar{m})^{t_{\bar{m}}}\}$. So  $\beta(R)$ avoids $111$ if and only if there are at most two of any element in this set which is equivalent to $s_1=1$, $s_i\leq 2$ for $i>1$, and $t_i\leq 2$ for all $i\geq 1$. Further this implies that $\beta(R)$ avoids $111$ if and only if  $Q\in \cN_{n-k-2}$ and $\beta(P)$ has exactly one $1$ and avoids $111$. Just as in our first case, $\beta(P)$ has exactly one $1$ and avoids $111$ if and only if $P=bP'$ for some $P'\in \cN_{k-1}$. 
Because $\area(R)=\area(P')+\area(Q)+k+1$ summing over all paths in this case gives us the term $q^{k+1}N_{k-1}(q)N_{n-k-2}(q)$ for $k>0$.
This completes the proof of the theorem.
\end{proof}

The next two avoidance classes can be characterized by a combination of Theorem~\ref{rgf:len3} and Lemma~\ref{1221structure}. The proofs are straightforward and  so  not included.

\begin{prop} We have  
$$R_n(112,1221)=\{12\dots m k^{n-m} :\ k\in [m]\}.$$ As such, for $n\geq 0$ we have
\begin{enumerate}
\item $\displaystyle F_n(112,1221)=\sum_{m=1}^n\sum_{k=1}^m q^{(n-m)(m-k)}r^{\binom{m}{2}+(n-m)(k-1)}s^{\binom{m}{2}}t^{m-k},$
\item $\displaystyle LB_n(112,1221)=\sum_{m=1}^n\sum_{k=1}^m q^{(n-m)(m-k)},$
\item $\displaystyle  \LS_n(112,1221)=\sum_{m=1}^n\sum_{k=1}^m q^{\binom{m}{2}+(n-m)(k-1)},$
\item $\displaystyle \RB_n(112,1221)=\sum_{m=1}^nm q^{\binom{m}{2}},$ and
\item $\displaystyle \RS_n(112,1221)=\sum_{i=1}^n iq^{n-i}.$\hqed
\end{enumerate}
\end{prop}

\begin{prop} We have $$R_n(123,1221)=\{1^n,11^{i}21^j2^{k}:\ i+j+k=n-2, \text{ and } i,j,k\ge0\}.$$
As such, for $n\geq 0$ we have, using the truth function $\chi(S)=1$ if $S$ is true or $0$ if $S$ is false,
\begin{enumerate}
\item $\displaystyle F_n(123,1221)=1+\sum_{i+j+k=n-2\atop i,j,k\ge 0} q^j r^{k+1}s^{i+1+j\cdot\chi(k>0)}t^{\chi(j>0)},$
\item $\displaystyle \LB_n(123,1221)=1+\sum_{j=0}^{n-2}(n-j-1)q^j,$
\item $\displaystyle  \LS_n(123,1221)=1+\sum_{k=0}^{n-2}(n-k-1)q^{k+1},$
\item $\displaystyle \RB_n(123,1221)=1+q^{n-1}+\sum_{k=1}^{n-2}(k+1)q^{k},$
and
\item $\displaystyle \RS_n(123,1221)=n+\binom{n-1}{2}q.$\hqed
\end{enumerate}
\label{R(123,1221)}
\end{prop}

\subsection{More about the pattern $1212$}

It turns out that the generating function $\LB_n(1212)$ is also equal to $M_{n-1}(q)$.  Instead of showing this directly, we prove that $\LB_n(1212)=\LB_n(1221)$ and then Corollary~\ref{LB(1221)=M(q)} completes the proof.  In the process we also show $\LS_n(1212)=\LS_n(1221)$.
\begin{prop}
The restriction $\inc:R_n(1212)\rightarrow R_n(1221)$ is a bijection which preserves $\lb$ and $\ls$. 
\label{bijection1221to1212}
\end{prop}

\begin{proof}
By Lemma~\ref{1221structure}  we have  $\lb(w)=\lb(\inc(w))$. This map also preserves $\ls$ because $w$ and $\inc(w)$ are rearrangements of each other and $\ls(w_i) = w_i-1$ for any RGF $w$.

Now we only need to show that $\inc:R_n(1212)\rightarrow R_n(1221)$ is bijective. Since $|R_n(1212)|=C_n=|R_n(1221)|$ it suffices to show the map is injective. Assume that $v=v_1v_2\dots v_n$ and $w=w_1\dots w_n$ are two distinct words which avoid $1212$, but $\inc(v)=\inc(w)$. This means that $v$ and $w$ share the same positions of first occurrences, and the same multiset of repeated letters. 
But since $v\neq w$ there is then a smallest index $i\geq 1$ such that 
$v_1\dots v_{i-1}=w_1\dots w_{i-1}$ but $v_i\neq w_i$. Without loss of generality let $v_i=x$, $w_i=y$, and $x<y$. We have noted that $v$ and $w$ have their first occurrences at the same indices, so $v_i$ and $w_i$ must be repeated letters. Since $w$ is an RGF, the first occurrence of $x$ and $y$ must occur before $w_i$, so $v$ also has the subword $xy$ before $v_i$. However, because $v$ and $w$ have the same collection of repeated letters and agree up to position $i-1$, the $y$ which is $w_i$ in $w$ must occur some time after $v_i$ in $v$. This means that $v$ has the subword $xyxy$ contradicting Lemma~\ref{noncross_avoid}. 
\end{proof}

\begin{cor} For $k\geq 0$ we have
$$F_n(1212;q,r,1,1)=F_n(1221;q,r,1,1),$$
$$F_n(1^k,1212;q,r,1,1)=F_n(1^k,1221;q,r,1,1),$$
and
$$F_n(12\dots k,1212;q,r,1,1)=F_n(12\dots k,1221;q,r,1,1).$$
\label{1221 and 1212 equalities}
\end{cor}
\begin{proof}
The bijection $f$ in Proposition~\ref{bijection1221to1212} preserves the number of times any integer  appears and preserves the maximum integer which appears. The equalities follow from this fact. 
\end{proof}

Using Corollary~\ref{rs1212=M}, 
Propositions~\ref{R(123,1212)} and~\ref{R(123,1221)}, and Corollaries~\ref{LB(1221)=M(q)} and~\ref{1221 and 1212 equalities} we have the following equalities which summarizing results in this section.

\begin{cor} 
\label{1212,1221}
We have, for $n\ge0$,
$$\LB_n(1212)=\RS_n(1212)=\LB_n(1221)=M_{n-1}(q),$$
$$\LS_n(1212)=\LS_n(1221),$$
$$\LB_n(111,1212)=\LB_n(111,1221),$$
$$\LS_n(111,1212)=\LS_n(111,1221),$$
$$\LB_n(123,1212)=\RS_n(123,1212)=\LB_n(123,1221),$$
and

\medskip

\eqed{\LS_n(123,1212)=\RB_n(123,1212)=\LS_n(123,1221).}
\end{cor}

We note that Simion~\cite{sim:csn} also proved $\LB_n(1212)=\RS_n(1212)$ by different means.  In addition, she showed the following.
\bth[\cite{sim:csn}]
\label{LSRB1212}
We have, for $n\ge0$,

\medskip

\eqed{LS_n(1212)=RB_n(1212).}
\eth


\section{Comments and open problems}
\label{cop}

We list some further possible lines of research in the hopes that the reader may be interested to pursue them.

\medskip

(1) {\bf Longer patterns.}  In Sections~\ref{rfl}, \ref{p1212}, and~\ref{p1221} we have begun the study of patterns of length four or more, but there are almost certainly more interesting results for such patterns.    For example, for noncrossing partitions it would be interesting to see if the polynomial $LS_n(1212)=RB_n(1212)$ can be viewed as the generating function for a statistic over two-colored Motzkin paths.  And here is a  specific conjecture for nonnesting patterns.

\bcon
The coefficients of $RB_n(1221)$ stabilize in the following sense.  Given $k$ there is a bound $N_k$ such that for $n\ge N_k$ the coefficient of $q^k$ in $\RB_n(1221)$ is constant.
\econ

\medskip

(2)  {\bf Vincular patterns.}  In the theory of permutation patterns a {\em vincular} or {\em generalized} pattern is one where copies of the pattern in a larger permutation are required to have certain elements adjacent.  One can indicate such elements by underlining them.  For example, a copy of the pattern $21$ is an inversion while a copy of the pattern $\ul{21}$ is a descent.  In~\cite{bs:gpp},  Babson and Steingr{\'{\i}}msson initiated the study of such patterns and showed that a wide array of  well-known permutation statistics could be realize as linear combinatorics of functions counting vincular patterns.  One can also consider patterns where certain integers which are numerically adjacent in the pattern must be numerically adjacent in the copy and indicate these by an overline.  So $\ol{21}$ would count inversions consisting of an element $k$ followed by $k-1$.  
And, of course, one could combine positional and numerical adjacency.
It seems probable that studying vincular RGF patterns would yield interesting enumerative results.

\medskip

(3)  {\bf Equidistribution.}  In their original paper, Wachs and White~\cite{ww:pqs} proved that $\lb$ and $\rs$ are equidistributed (have the same generating function)  over the set of all RGFs of length $n$ with maximum $m$.   They also showed that $\ls$ and $\rb$ are equidistributed over the same set of RGFs.  We have seen similar behavior in Theorems~\ref{connectLB112}, \ref{connectLS112RB122}, \ref{RS122LBRS123}, and~\ref{LSRB1212} as well as  Corollaries~\ref{sym} and~\ref{1212,1221}.  It would be very interesting to derive some of these results from  more general theorems which would guarantee equidistribution for a large number of avoidance classes.

\medskip

(4)  {\bf  Mahonian pairs.}  When considering $\st$-Wilf equivalence, one has a single statistic which  has  the same generating function over two different avoidance classes.  When considering equidistribution, one has two different statistics which have the same generating function over a given avoidance class.  Obviously, one could generalize both notions by considering one statistic on an avoidance class and a second statistic on another class.  For the permutation statistics given by the major index, $\maj$, and inversion number, $\inv$, this concept was first studied by Sagan and Savage~\cite{ss:mp}.  Such pairs of statistics and classes were called {\em Mahonian pairs} since $\maj$ and $\inv$ both have the Mahonian distribution over the full symmetric group.  In the present work, we have found such equalities in the results cited in (3) as well as in Theorem~\ref{connectLS122} and
Corollary~\ref{1221 and 1212 equalities}.   Again, a more general explanation of when such identities occur would be desirable. 

\medskip

(5) {\bf Other statistics.}  There are other statistics related to the four we have been studying.  Given an integer sequence $w=w_1\dots w_n$, Simion and Stanton~\cite{ss:olp} considered a statistic counting smaller elements both to the left and the right of each $w_j$ by letting
$$
\lrs(w_j) =\#\{x<w_j\ :\ \text{there are $i<j<k$ with $w_i=w_k=x$}\}
$$
and $\lrs(w)=\sum_j \lrs(w_j)$.  Note that if $w$ is an RGF then $\lrs(w)=\rs(w)$.  They also looked at an analogous statistic for counting bigger elements, as well as refinements of both statistics obtained by restricting them to certain elements of $w$ related to first occurrences and repeated elements.  Their motivation came from studying a generalization of the Laguerre polynomials.  In the process, they obtained results about these statistics on noncrossing and nonnesting  RGFs.  It would be interesting to investigate these statistics in relation to other patterns.

\medskip

{\em Acknowledgment.} We would like to thank Anisse Kasraoui and Dennis Stanton for interesting comments and important references.


\newcommand{\etalchar}[1]{$^{#1}$}


\begin{thebibliography}{DDG{\etalchar{+}}16}

\bibitem[Arm09]{arm:gnp}
Drew Armstrong.
\newblock Generalized noncrossing partitions and combinatorics of {C}oxeter
  groups.
\newblock {\em Mem. Amer. Math. Soc.}, 202(949):x+159, 2009.

\bibitem[BBES14]{bbes:dsi}
Marilena Barnabei, Flavio Bonetti, Sergi Elizalde, and Matteo Silimbani.
\newblock Descent sets on 321-avoiding involutions and hook decompositions of
  integer partitions.
\newblock {\em J. Combin. Theory Ser. A}, 128(1):132--148, November 2014.

\bibitem[BS]{bs:pas}
Jonathan Bloom and Dan Saracino.
\newblock Pattern avoidance for set partitions \`a la klazar.
\newblock Preprint {\texttt{arXiv:1511.00192v2}}.

\bibitem[BS00]{bs:gpp}
Eric Babson and Einar Steingr{\'{\i}}msson.
\newblock Generalized permutation patterns and a classification of the
  {M}ahonian statistics.
\newblock {\em S\'em. Lothar. Combin.}, 44:Art. B44b, 18 pp., 2000.

\bibitem[CDD{\etalchar{+}}13]{cddds:zas}
William Y.~C. Chen, Alvin Y.~L. Dai, Theodore Dokos, Tim Dwyer, and Bruce~E.
  Sagan.
\newblock On 021-avoiding ascent sequences.
\newblock {\em Electron. J. Combin.}, 20(1):Paper 76, 6, 2013.

\bibitem[CEKS13]{ceks:ipt}
Szu-En Cheng, Sergi Elizalde, Anisse Kasraoui, and Bruce~E. Sagan.
\newblock Inversion polynomials for 321-avoiding permutations.
\newblock {\em Discrete Math.}, 313(22):2552--2565, 2013.

\bibitem[DDG{\etalchar{+}}16]{ddggprs:spp}
Samantha Dahlberg, Robert Dorward, Jonathan Gerhard, Thomas Grubb, Carlin
  Purcell, Lindsey Reppuhn, and Bruce~E. Sagan.
\newblock Set partition patterns and statistics.
\newblock {\em Discrete Math.}, 339(1):1--16, 2016.

\bibitem[DDJ{\etalchar{+}}12]{ddjss:pps}
Theodore Dokos, Tim Dwyer, Bryan~P. Johnson, Bruce~E. Sagan, and Kimberly
  Selsor.
\newblock Permutation patterns and statistics.
\newblock {\em Discrete Math.}, 312(18):2760--2775, 2012.

\bibitem[Dra09]{dra:laulp}
Brian Drake.
\newblock Limits of areas under lattice paths.
\newblock {\em Discrete Math.}, 309(12):3936--3953, 2009.

\bibitem[DS11]{ds:paa}
Paul Duncan and Einar Steingr{\'{\i}}msson.
\newblock Pattern avoidance in ascent sequences.
\newblock {\em Electron. J. Combin.}, 18(1):Paper 226, 17, 2011.

\bibitem[GM09]{gm:psq}
Adam~M. Goyt and David Mathisen.
\newblock Permutation statistics and {$q$}-{F}ibonacci numbers.
\newblock {\em Electron. J. Combin.}, 16(1):Research Paper 101, 15 pp., 2009.

\bibitem[Goy08]{goy:apt}
Adam~M. Goyt.
\newblock Avoidance of partitions of a three-element set.
\newblock {\em Adv. in Appl. Math.}, 41(1):95--114, 2008.

\bibitem[GS09]{gs:sps}
Adam~M. Goyt and Bruce~E. Sagan.
\newblock Set partition statistics and {$q$}-{F}ibonacci numbers.
\newblock {\em European J. Combin.}, 30(1):230--245, 2009.

\bibitem[JM08]{jm:pap}
V\'{i}t Jel\'{i}nek and Toufik Mansour.
\newblock On pattern-avoiding partitions.
\newblock {\em Electron J. Combin.}, 15(R39):1--52, 2008.

\bibitem[Kil12]{kil:pcc}
Kendra Killpatrick.
\newblock On the parity of certain coefficients for a {$q$}-analogue of the
  {C}atalan numbers.
\newblock {\em Electron. J. Combin.}, 19(4):Paper 27, 7, 2012.

\bibitem[Kla96]{kla:aas}
Martin Klazar.
\newblock On {$abab$}-free and {$abba$}-free set partitions.
\newblock {\em European J. Combin.}, 17(1):53--68, 1996.

\bibitem[Kla00a]{kla:cpsI}
Martin Klazar.
\newblock Counting pattern-free set partitions. {I}. {A} generalization of
  {S}tirling numbers of the second kind.
\newblock {\em European J. Combin.}, 21(3):367--378, 2000.

\bibitem[Kla00b]{kla:cpsII}
Martin Klazar.
\newblock Counting pattern-free set partitions. {II}. {N}oncrossing and other
  hypergraphs.
\newblock {\em Electron. J. Combin.}, 7:Research Paper 34, 25 pp., 2000.

\bibitem[MS11]{ms:paspcn}
Toufik Mansour and Mark Shattuck.
\newblock Pattern-avoiding set partitions and {C}atalan numbers.
\newblock {\em Electron. J. Combin.}, 18(2):Paper 34, 18, 2011.

\bibitem[Sag10]{sag:pas}
Bruce~E. Sagan.
\newblock Pattern avoidance in set partitions.
\newblock {\em Ars Combin.}, 94:79--96, 2010.

\bibitem[Sim94]{sim:csn}
Rodica Simion.
\newblock Combinatorial statistics on noncrossing partitions.
\newblock {\em J. Combin. Theory Ser. A}, 66(2):270--301, 1994.

\bibitem[SS96]{ss:olp}
R.~Simion and D.~Stanton.
\newblock Octabasic {L}aguerre polynomials and permutation statistics.
\newblock {\em J. Comput. Appl. Math.}, 68(1-2):297--329, 1996.

\bibitem[SS12]{ss:mp}
Bruce~E. Sagan and Carla~D. Savage.
\newblock Mahonian pairs.
\newblock {\em J. Combin. Theory Ser. A}, 119(3):526--545, 2012.

\bibitem[Vie85]{vie:ctg}
G{\'e}rard Viennot.
\newblock A combinatorial theory for general orthogonal polynomials with
  extensions and applications.
\newblock In {\em Orthogonal polynomials and applications ({B}ar-le-{D}uc,
  1984)}, volume 1171 of {\em Lecture Notes in Math.}, pages 139--157.
  Springer, Berlin, 1985.

\bibitem[WW91]{ww:pqs}
Michelle Wachs and Dennis White.
\newblock {$p,q$}-{S}tirling numbers and set partition statistics.
\newblock {\em J. Combin. Theory Ser. A}, 56(1):27--46, 1991.

\end{thebibliography}
\end{document}